\documentclass[11pt]{article}
\usepackage[margin=1in]{geometry}

\usepackage{xspace}
\usepackage{xcolor}
\usepackage{amscd}
\usepackage{amsmath}
\usepackage{amssymb}
\usepackage{amstext}
\usepackage{amsthm}
\usepackage{thm-restate}
\usepackage{bm}
\usepackage{colonequals}
\usepackage{tikz}
\usepackage{booktabs}
\usepackage{array}
\usepackage{tabularx}
\usepackage{makecell}
\usepackage{enumitem}
\usepackage[pagebackref=true,colorlinks,urlcolor=blue,linkcolor=purple,citecolor=blue,bookmarks,bookmarksopen,bookmarksnumbered]{hyperref}
\usepackage[capitalise,nameinlink]{cleveref}
\usepackage{subfigure}

\usepackage{mathtools}
\usepackage{amssymb, amsfonts, amsthm}

\usepackage{array}
\usepackage{verbatim}
\usepackage{mathrsfs}
\usepackage{graphicx}
\usepackage{algorithm}
\usepackage{algpseudocode}
\usepackage{caption}



\newcommand*{\alglineref}[1]{\hyperref[#1]{Line~\ref*{#1}}}

\usepackage{tikz}
\usetikzlibrary{decorations.pathreplacing}

\usepackage[toc,page]{appendix}

\usepackage{titlesec}
\titleformat{\subsubsection}[runin]
{\normalfont\normalsize\bfseries}
{\thesubsubsection}{1em}{}

\theoremstyle{plain}
\newtheorem{theorem}{Theorem}[section]
\newtheorem{lemma}[theorem]{Lemma}
\newtheorem{proposition}[theorem]{Proposition}
\newtheorem{corollary}[theorem]{Corollary}

\theoremstyle{definition}
\newtheorem{definition}[theorem]{Definition}

\newcommand{\Z}{\mathbb{Z}}

\newcommand{\R}{\mathbb{R}}

\DeclareSymbolFont{bbold}{U}{bbold}{m}{n}
\DeclareSymbolFontAlphabet{\mathbbold}{bbold}

\newcommand{\supp}{\mathrm{supp}}

\renewcommand{\epsilon}{\varepsilon}

\newcommand{\ind}{\operatorname{ind}}
\newcommand{\conv}{\operatorname{conv}}
\newcommand{\len}{\mathsf{len}}
\newcommand{\wind}{\mathsf{wind}}
\newcommand{\sd}{\mathsf{sd}}
\newcommand{\id}{\mathsf{id}}
\newcommand{\diam}{\mathsf{diam}}

\newif\ifshownotes
\shownotestrue

\ifshownotes
	\newcommand{\hstnote}[1]{\textcolor{red}{\textbf{[Shengtang: #1]}}}
    \newcommand{\xhy}[1]{\textcolor{orange}{\textbf{[xhy: #1]}}}
\else
	\newcommand{\hstnote}[1]{}
    \newcommand{\xhy}[1]{}
\fi

\numberwithin{equation}{section}

\newcommand{\myparskip}{2pt}
\parskip \myparskip
\usepackage{setspace}

\usepackage{titling}

\begin{document}

\begin{titlepage}
\title{Improved Lower Bounds on Tree Covers}

\author{Yu Chen\thanks{National University of Singapore, Singapore. Email: {\tt yu.chen@nus.edu.sg}.}
\and Shengtang Huang\thanks{CFCS, School of Computer Science, Peking University, Beijing, China. Email: {\tt peanuttang1320061044@gmail.com}.}
\and Zihan Tan\thanks{University of Minnesota Twin Cities, MN, USA. Email: {\tt ztan@umn.edu}.}
\and Hangyu Xu\thanks{University of Science and Technology of China, Hefei, China. Email: {\tt hangyuxu@mail.ustc.edu.cn}.}
}

\maketitle

\begin{abstract}
A tree cover of an $n$-point metric space is a collection of $k$ trees such that every pairwise distance is approximately preserved in at least one tree. For fixed $k$, the best known distortion upper bound is $\widetilde{O}(n^{1/k})$ and the best known lower bound is $\Omega_k(n^{1/2^{k-1}})$.

In this paper, we prove a lower bound of $\Omega_k(n^{1/(k-1)^2})$ for all $k\ge 3$, significantly narrowing the gap towards the upper bound.

Our approach makes use of a novel algebraic topological notion that we call the \emph{agreement index}: the equivariant topological index of a simplicial complex encoding pairwise agreement among tree labels. This notion helps characterize the power and limit of algebraic and topological methods for analyzing tree covers, and we believe it will also find applications to other problems.
\end{abstract}

\thispagestyle{empty}

\end{titlepage}

\newpage

\thispagestyle{empty}
\begin{spacing}{0.8}
    \tableofcontents
\end{spacing}

\newpage
\setcounter{page}{1}

\section{Introduction}
\label{sec:intro}

For an $n$-point metric space $(X,d_X)$, an \emph{$(\alpha,k)$-tree cover} is a set of $k$ edge-weighted trees $T_1,\ldots,T_k$ on vertex set $X$ such that, for every $x,y\in X$,
\begin{itemize}
    \item for every $i\in[k]$, $d_{T_i}(x,y)\ge d_X(x,y)$ (i.e., $T_i$ is a \emph{dominating} tree); and
    \item there exists $i\in[k]$ such that $d_{T_i}(x,y)\le \alpha\cdot d_X(x,y)$.
\end{itemize}
Tree covers combine the simple structure of tree metrics with the flexibility to choose a tree for each pair. They arise in routing~\cite{AP92,CGMZ16}, network design~\cite{GHR06}, distance oracles~\cite{MN07}, and metric Ramsey theory~\cite{BLMN03,BFN22}. The parameters $k$ and $\alpha$ are the \emph{size} and \emph{distortion} of the cover, respectively. The central problem is to understand the trade-off between them.

For general metrics, the size-distortion trade-off remains open even when $k$ is a small constant. When $k=1$, the worst-case distortion is $\Theta(n)$~\cite{RR98}. For larger $k$, Bartal, Fandina, and Neiman~\cite{BFN22} established an upper bound $\widetilde O(n^{1/k})$ and a lower bound $\Omega(\log_k n)$. Recently, Chen, Tan, and Xu~\cite{CTX26} improved the lower bound to $\Omega_k\big(n^{1/2^{k-1}}\big)$. Their result matches (nearly) the $\tilde O(\sqrt{n})$ upper bound by \cite{BFN22} for $k=2$, and leaves open a gap ($1/k$ vs $1/2^{k-1}$) in the exponent for larger $k$.

In addition to improving the bounds, \cite{CTX26} also offers a novel approach that connects tree covers to combinatorial topology. The proof constructs a discretized ball with antipodal boundary identifications and assigns binary labels to its vertices using the trees. Tucker's lemma is then utilized to guarantee the existence of two nearby points that are far apart in every tree. The dimension of the construction, $2^{k-1}$, determines the exponent in the resulting lower bound. Therefore, the natural next question is:
\[
\emph{Can topological methods yield a substantially better bound, and what are their inherent limitations?}
\]

\subsection{Our Results}

In this work, we further explore the topological approach for proving lower bounds on tree covers. Our main result is the following theorem.

\begin{theorem}
\label{thm:intro-main}
For every integer $k\ge 3$ and every sufficiently large integer $n$, there exists an $n$-point metric space for which every size-$k$ tree cover has distortion at least $\Omega_k\big(n^{1/(k-1)^2}\big)$.
\end{theorem}

Our bound improves upon the previous bound $\Omega_k\big(n^{1/2^{k-1}}\big)$ of \cite{CTX26} for all $k\ge 6$. Compared with the upper-bound exponent $1/k$, the remaining gap in the exponents is a factor of $(k-1)^2/k=O(k)$.

Next, to understand how much further this approach can go, we introduce a topological notion called the \emph{agreement index}, that governs the strength of our construction: upper bounds on the index yield distortion lower bounds, while lower bounds on the index constrain the exponent obtainable through this framework. Our main theorem essentially follows from a quadratic upper bound on the agreement index. Conversely, we also show that this index must grow superlinearly with $k$, yielding the following limitation for approaches captured by our framework.

\begin{proposition}[Informal]
\label{prop: superlinear}
Similar algebraic/topological approaches cannot yield a better distortion lower bound than $n^{O(1/k^{\log_3 4})}$.
\end{proposition}

We will provide more details on agreement index in \Cref{sec:metric-transfer} and \Cref{sec:index-upper-bound} and give a more formal statement of \Cref{prop: superlinear} in \Cref{sec:lower_bound} and \Cref{sec:reduction}.

\subsection{Related Work}

\paragraph{Ramsey tree covers.}
A Ramsey tree cover requires, additionally, that each point chooses a single tree that approximates its distances to all other points. The upper bound $\widetilde O(n^{1/k})$~\cite{BFN22} satisfies this stronger condition, and is near optimal for Ramsey tree covers: there are metrics requiring distortion $\Omega_k(n^{1/k})$~\cite{BFN22}. Chen, Tan, and Xu~\cite{CTX26} gave another proof using a topological construction.

\paragraph{Structured metrics.}
For several important metric families, a number of trees independent of $n$ suffices for constant distortion. In Euclidean spaces of fixed dimension, the Dumbbell Theorem~\cite{ADMSS95} gives covers with distortion $1+\varepsilon$ and size depending only on the dimension and $\varepsilon$; the dependence on $\varepsilon$ was subsequently improved in~\cite{CCLMST24_Euclidean}. For the Euclidean plane, constant distortion is achievable with only two trees~\cite{BKT26,CTX26,LMSZ26}. For planar and fixed-minor-free graph metrics, covers with distortion $1+\varepsilon$ and size independent of $n$ are known in the variant that permits additional Steiner vertices in the trees~\cite{CCLMS23,CCLMST24_Shortcut,le2026three}. These results highlight the different behavior of structured metrics and the general metrics considered here.

\subsection{Organization}

We start in \Cref{sec:overview} with a technical overview. Then in \Cref{sec:preli}, we present topological preliminaries. In \Cref{sec:metric-transfer} and \Cref{sec:index-upper-bound}, we introduce the central notion \emph{agreement index}, and show how it implies the lower bound on tree cover in \Cref{thm:intro-main}. Finally, in \Cref{sec:lower_bound} and \Cref{sec:reduction}, we establish the lower bound on agreement index, which implies the limitation of this topological framework, giving a proof of \Cref{prop: superlinear}.

\section{Technical Overview}
\label{sec:overview}

In this section, we provide a high-level overview of our approach. We begin by reviewing the previous approach of \cite{CTX26}.

\subsection{The Approach of \cite{CTX26}}

\cite{CTX26} construct a metric from the integer lattice points in the $\ell_1$-ball $B_r^D=\{x\in\mathbb Z^D:\|x\|_1\le r\}$. Ordinary edges connect points at $\ell_1$-distance one and have weight one. Special edges connect antipodal boundary points and have weight zero. Contracting the special edges gives a metric with $n=\Theta_D(r^D)$ points. The crux of their analysis is to assign labels to points according to their positions in the trees, which allows them to apply Tucker's lemma to find two nearby points that are far apart in every tree.

\paragraph{From tree paths to binary labels.}
Root each tree $T$ at the origin $\vec{0}$. Each metric point corresponds to either one lattice point or two antipodal boundary points. Choose (arbitrarily) one of these lattice points as its representative. For each tree edge, fix a shortest walk in the original graph between the representatives of its endpoints, which may use both ordinary and special edges. To obtain a walk from $\vec{0}$ to a lattice point $x$, follow the tree path from the root to the corresponding metric point, replacing each tree edge by its chosen walk. The resulting walk ends at the chosen representative, which is either $x$ itself or its antipodal boundary point $-x$. If it ends at $-x$, take one additional zero-weight special edge to reach $x$. We call this walk the \emph{lifted root path} to $x$.

Let $s_T(x)$ be the number of special-edge traversals along this walk, and define $\ell_T(x)=(-1)^{s_T(x)}$. For a boundary point $x$, the lifted root paths to $x$ and $-x$ differ by exactly one final special edge. Their parities are therefore opposite, giving $\ell_T(-x)=-\ell_T(x)$ on the boundary.

They then show that if a pair $u,v$ of points are \emph{(i) close in the original $\ell_1$-ball, and (ii) having opposite labels in $T$}, then they are far from each other in $T$.

On the one hand, the opposite labels give a $u$-$v$ walk with an odd number of special edges. To obtain this walk, follow the lifted root path to $u$ backwards, then follow the lifted root path to $v$ forwards. The common portion of the two tree paths is traversed in both directions and can be canceled. This cancellation removes an even number of special-edge traversals, so the remaining walk still contains an odd number of special edges. Its length is at most $d_T(u,v)$, since each remaining tree edge is replaced by a shortest graph walk of no greater length.

On the other hand, such a walk must be long. To see this, remove its antipodal jumps one at a time: whenever the walk traverses a special edge from $x$ to $-x$, reflect the remaining portion through the origin. This removes the jump and preserves the lengths of all ordinary edges. Since there are an odd number of jumps, the resulting walk ends at $-v$ rather than $v$. It is therefore an ordinary walk from $u$ to $-v$ with the same length. Since it passes through a boundary point $z$ where a jump was removed, $d_T(u,v) \ge \|u-z\|_1+\|z+v\|_1 \ge 2r-\|u-v\|_1$, which is $\Omega(r)$ if $\|u-v\|_1=O(1)$.

\paragraph{Tucker's lemma finds the witness pair.}
For a tree cover of $k$ trees $\{T_1,\ldots,T_k\}$, we can combine the labels for all the trees into a word
\[
    \ell(x)=\bigl(\ell_{T_1}(x),\ldots,\ell_{T_k}(x)\bigr)
    \in\{-1,+1\}^k.
\]
The same pair $u,v$ has different labels in every tree precisely when $\ell(u)=-\ell(v)$. Our goal is therefore to \emph{find two nearby vertices with opposite words}.

There are $2^{k-1}$ pairs of opposite words. Set $D=2^{k-1}$, number these pairs from $1$ to $D$, and assign the pair of words in the $j$-th pair the labels $+j$ and $-j$. This gives labels in $\{\pm1,\ldots,\pm D\}$, with antipodal boundary points receiving opposite labels. Choose a triangulation of the ball with antipodally symmetric boundary and $O_D(1)$  simplex diameter. Tucker's lemma guarantees a complementary edge: an edge whose endpoints receive complementary labels $+j$ and $-j$. The endpoints $u,v$ of this edge are at distance $O_k(1)$ in the original graph but have opposite words. From previous discussion, $u,v$ are at distance $\Omega_k(r)$ in every tree. Since $n=\Theta_k(r^D)$ and $D=2^{k-1}$, the distortion from $u,v$ is $\Omega_k\big(n^{1/2^{k-1}}\big)$.

\subsection{From Antipodal to Cyclic Symmetry: Already an $n^{1/O(k^2)}$ Lower Bound}
\label{sec: cyclic}

A natural observation of the \cite{CTX26} approach is that their binary labels are too restrictive in producing a witness pair, as every word in $\{-1,+1\}^k$ has exactly one opposite word. Our first modification is to use labels in $\Z_p$ instead of binary labels (which can be viewed as labels in $\Z_2$), so words in $\Z_p^k$ now have $(p-1)^k$ partners differing in every coordinate. As it turns out, guaranteeing a witness pair in this case becomes much easier and therefore requires a lower dimension $D$ in the original ball construction, leading to a larger value of $r$ and a stronger distortion lower bound.

\paragraph{An equivariant Borsuk-Ulam Theorem.}
Instead of using Tucker's lemma which handles binary labels, we need a topological tool generalizes the antipodal $\Z_2$ symmetry to a cyclic $\Z_p$ symmetry. We use the following special case of a theorem of Musin and Volovikov~\cite[Corollary~5.1]{MV_G_topo}. Recall that an action is \emph{free} if no nonidentity group element fixes a point.

\begin{theorem}
\label{thm:MV_warmup}
Let $B$ be a $D$-dimensional ball equipped with a $\Z_p$-action that is free on $\partial B$. Suppose that $\Z_p$ acts linearly on $\mathbb R^D$ and freely on $\mathbb R^D\setminus\{0\}$. Then every continuous map $f:B\to\mathbb R^D$ satisfying $f(gx)=g f(x)$ for every $g\in\Z_p$ and $x\in\partial B$ has a zero.
\end{theorem}

Note that when $p=2$, with the antipodal action on $B$, this is the standard ``zero formulation'' of the Borsuk-Ulam theorem: a continuous map that is odd on the boundary must have a zero.

\paragraph{Cyclic identifications and winding labels.}
The above equivariant Borsuk-Ulam Theorem gives clear guidance on our construction. Instead of identifying boundary points in antipodal pairs, we now identify them in groups of $p$. In order to organize these identifications, we seek a symmetry $\gamma$ of the ball such that each boundary point belongs to a cycle of $p$ distinct points $x,\gamma x,\ldots,\gamma^{p-1}x$. We connect consecutive points in each cycle by zero-weight special edges and contract the cycle to a single metric point. Ordinary edges still connect nearby lattice points and have weight one.

We lift tree paths to this graph as before. Instead of focusing on the parity of boundary jumps in the $\Z_2$ case, we orient each boundary cycle in the direction \(x\to\gamma x\), and count each forward step around a special cycle as $+1$ and each backward step as $-1$. The total modulo $p$ is the \emph{winding} of the walk. For each tree $T$, the winding of the lifted root path defines a label $\ell_T(x)\in\Z_p$. Moving one step along a boundary cycle changes the label by one, giving $\ell_T(\gamma x)=\ell_T(x)+1$ on the boundary. Cancellation of common root paths shows that the lifted $u$-$v$ tree path has winding $\ell_T(v)-\ell_T(u)$. Thus, different labels again certify a nontrivial pattern of boundary jumps.

We now choose the vector representation of the labels and the geometry of the ball so that the above equivariant Borsuk-Ulam theorem applies. Represent the $p$ possible labels by the vertices $q_0,\ldots,q_{p-1}$ of a regular simplex centered at the origin in $\mathbb R^{p-1}$. Concatenating the vectors representing a point's labels in the $k$ trees gives a vector in $\mathbb R^D$, where $D=k(p-1)$. Each tree contributes a separate group of $p-1$ coordinates, which we call a \emph{block}.

The boundary labels transform by adding one modulo $p$. To express this transformation geometrically, define $\gamma$ by applying, in every block, the linear symmetry that sends $q_a$ to $q_{a+1}$. The theorem requires the resulting action to be free away from the origin. We ensure this by requiring $p$ to be prime: every nontrivial power $\gamma^a$ then permutes all $p$ simplex vertices in a single cycle, so its only fixed point in each block is the simplex center. Thus, no nontrivial power of $\gamma$ fixes a nonzero vector in $\mathbb R^D$.

We use this same $D$-dimensional space for the metric construction, taking a radius-$r$ ball $B_r$ centered at the origin $\vec{0}$ and choosing the ball and its lattice discretization to be preserved by $\gamma$. The action is then free on $\partial B_r$ as well. We triangulate $B_r$ into simplices of diameter $O_{p,k}(1)$, with $\gamma$ mapping every boundary simplex to another boundary simplex. At each triangulation vertex $v$, define
\[
    F(v)=
    \bigl(q_{\ell_{T_1}(v)},\ldots,q_{\ell_{T_k}(v)}\bigr)
    \in\mathbb R^D,
\]
and extend $F$ affinely over each simplex. These extensions agree on shared faces, giving a continuous map $F:B_r\to\mathbb R^D$. It remains to verify boundary equivariance. The winding labels satisfy $\ell_{T_i}(\gamma v)=\ell_{T_i}(v)+1$ at every boundary vertex $v$. By our choice of the action on the label vectors, this gives
\[
    F(\gamma v)
    =
    \bigl(q_{\ell_{T_1}(v)+1},\ldots,
          q_{\ell_{T_k}(v)+1}\bigr)
    =\gamma F(v).
\]
Since $\gamma$ preserves the boundary triangulation and acts linearly, the affine extension satisfies the same identity at every point of $\partial B_r$. Thus, $F$ is continuous and equivariant on the boundary, and the actions on both $\partial B_r$ and $\mathbb R^D\setminus\{0\}$ are free. The theorem therefore supplies a point $x^*\in B_r$ with $F(x^*)=0$.

\paragraph{A balanced simplex supplies a witness pair.}
Let $\sigma$ be a simplex containing $x^*$. Write $x^*$ as a convex combination of its vertices gives weights $\lambda_v\ge0$ with $\sum_{v\in V(\sigma)}\lambda_v=1$. As $F$ is affine on $\sigma$,
\[
    0=F(x^*)=\sum_{v\in V(\sigma)}\lambda_v F(v).
\]
These weights define a probability distribution on the vertices of $\sigma$. Looking at the block corresponding to tree $T_i$, we obtain $\sum_{v\in V(\sigma)} \lambda_v q_{\ell_{T_i}(v)}=0$. The center of a regular simplex has a unique expression as a convex combination of its vertices, assigning weight $1/p$ to each vertex. Hence
\[
    \sum_{\substack{v\in V(\sigma)\\ \ell_{T_i}(v)=a}}
        \lambda_v=\frac1p
    \qquad\text{for every }i\in[k]\text{ and }a\in\Z_p.
\]
Thus, under this single distribution, the label in every tree is uniform on $\Z_p$. Sample two vertices $u,u$ independently from this distribution. Their labels agree in any specified tree with probability $1/p$, so a union bound gives
\[
    \Pr\bigl[\ell_{T_i}(u)=\ell_{T_i}(v)
             \text{ for some }i\in[k]\bigr]
    \le \frac{k}{p}.
\]
To make this upper bound smaller than one, we require $p>k$. Therefore, we obtain a pair of points in the same small simplex whose labels differ in every tree, which we call the \emph{witness pair}.

\paragraph{From the witness pair to distortion.}
Let $u,v$ be the pair obtained above. Since they lie in the same small simplex, their distance in the original graph, and hence in the quotient metric, is $O_{p,k}(1)$. We show that their different labels force them to be at distance $\Omega_{p,k}(r)$ in every tree.

Fix a tree $T$ and let $a=\ell_T(v)-\ell_T(u)\ne0$ in $\Z_p$. The lifted $u$-$v$ tree path has winding $a$ and length at most $d_T(u,v)$. As in the antipodal argument, we can remove each special jump by applying the inverse cyclic symmetry to the remaining walk. This preserves the lengths of ordinary edges and produces an ordinary walk from $u$ to $\gamma^{-a}v$, still of length at most $d_T(u,v)$.

To see why this walk is long, first suppose that one endpoint (say $v$) lies in the outer half of the ball. The freeness property used above now provides separation: since every nontrivial power of $\gamma$ fixes only the origin, $v$ is separated from $\gamma^{-a}v$ by $\Omega_{p,k}(r)$. Since points $u,v$ are within distance $O_{p,k}(1)$ from each other, the ordinary walk from $u$ to $\gamma^{-a}v$ has length $\Omega_{p,k}(r)$ for sufficiently large $r$. If both endpoints lie in the inner half, the nonzero winding guarantees that the lifted walk traverses a special edge and therefore visits the boundary. Reaching the boundary from the inner half costs $\Omega(r)$. In either case, $d_T(u,v)=\Omega_{p,k}(r)$. This holds for every tree, so the witness pair forces distortion $\Omega_{p,k}(r)$.

To analyze the distortion of the witness pair, it remains to choose $p$. To make the chosen cyclic action free, $p$ needs to be a prime number, and to obtain a witness pair from the sampling argument we need $p>k$. Therefore, we take $p$ to be the smallest prime greater than $k$, and so we know that $p=O(k)$. The construction therefore has dimension $D=k(p-1)=O(k^2)$. Since it contains $n=\Theta_k(r^D)$ metric points, the resulting distortion lower bound is $\Omega_k\big(n^{1/[k(p-1)]}\big)=\Omega_k\big(n^{O(1/k^2)}\big)$.

\subsection{Refining the Dimension via the Agreement Index:
the $\Omega_k\big(n^{1/(k-1)^2}\big)$ Bound}

After replacing antipodal by cyclic semmetry, the previous argument finds a witness pair by locating a balanced simplex and then applying a probabilistic argument. This approach is simple and clean, but the balanced simplex requirement may impose a higher dimension than necessary. To further improve the distortion lower bound to $\Omega_k\big(n^{1/(k-1)^2}\big)$, we give a more precise characterization of what prevents two nearby points from having labels that differ in every coordinate: the \emph{agreement complex}.

\paragraph{The agreement complex.}
If a simplex contains no witness pair, then every pair of label words on its vertices agree in at least one coordinate. We collect all such sets into the \emph{agreement complex} $I_{p,k}$: its vertices are the words in $\Z_p^k$, and its faces are the subsets of words in which every two words agree in some coordinates. Such a complex comes with a natural geometric realization: a point of its geometric realization $|I_{p,k}|$ can be viewed as a probability distribution supported on a pairwise-agreeing subset.

Similar to \Cref{sec: cyclic}, we define the cyclic action on words: $+1$ to all coordinates. This action is free on its realization: a fixed point would have support containing both a word $w$ and $w+t\cdot\mathbf{1}$ for some integer $t\in \{1,\ldots,p-1\}$, which disagree in every coordinate. This brings us to the real power of the agreement complex, which is that \emph{it turns the search for a witness pair into a topological obstruction}: if no small simplex contains a witness pair, its vertex labels form a face of $I_{p,k}$, and then the winding labels therefore extend affinely to a continuous map into $|I_{p,k}|$, equivariant on the boundary. We will choose the geometric construction so that such a map cannot exist, thereby forcing a witness pair.

\paragraph{Dold's theorem, the agreement index, and the hard instance construction.}
To obtain the desired topological obstruction, the new tool we will use is Dold's theorem~\cite{Dold83}, which generalizes the Borsuk-Ulam theorem from spheres with antipodal symmetry to spaces with free group actions.

We take $d+1$ disjoint copies of $\Z_p$ and form a simplicial complex whose faces contain $\le 1$ vertex from each copy. Denote the resulting $d$-dimensional complex by $E_d\Z_p$. Simultaneous $+1$ in all copies gives a free action. Its \emph{cone}, denoted by $C(E_d\Z_p)$, is obtained from $E_d\Z_p$ by adding a new point, called the \emph{apex}, and joining it to each point of $|E_d\Z_p|$ by a line segment. The original $|E_d\Z_p|$ forms the \emph{base} of the cone, and moving along these segments towards the apex continuously contracts the base to a single point. The cone thus has dimension $d+1$. See \Cref{fig:hard-instances} for an illustration and comparison. We use the following equivalent formulation of Dold's theorem:
\[
\emph{there is
no continuous map
$C(E_d\Z_p)\longrightarrow |E_d\Z_p|$
that is equivariant on the base $|E_d\Z_p|$}.
\]

For $p=2$, $|E_d\Z_2|$ is a $d$-dimensional sphere with antipodal action, and its cone is a $(d+1)$-dimensional ball. The above statement then rules out a continuous map from the ball to its boundary sphere that is odd on the boundary, which is an equivalent formulation of the Borsuk-Ulam theorem. For larger $p$, these complexes provide the corresponding obstruction for cyclic symmetry, which is what we need.

\begin{figure}[!htbp]
    \centering
    \begingroup
    \definecolor{orbitblue}{RGB}{0,114,178}
    \definecolor{orbitgreen}{RGB}{0,139,111}
    \definecolor{orbitpurple}{RGB}{125,75,155}
    \definecolor{reverseorange}{RGB}{195,83,25}
    \tikzset{
        instance geometry/.style={draw=black!40,line width=0.45pt},
        instance guide/.style={draw=black!20,line width=0.4pt},
        instance special/.style={line width=1.15pt,line cap=round},
        instance forward/.style={instance special,->,>=latex,shorten <=2pt,shorten >=2pt},
        instance reverse/.style={draw=reverseorange,line width=0.9pt,dashed,->,>=latex,shorten <=2pt,shorten >=2pt}
    }

    \subfigure[Ball $B^3\cong C(S^2)$, with antipodals.]{\begin{minipage}[t]{0.32\linewidth}
        \centering
        \begin{tikzpicture}[x=1.1cm,y=1.1cm,font=\small]
            \path[use as bounding box] (-2.05,-3.35) rectangle (2.05,1.85);
            \fill[black!2] (0,0) circle (1.62);
            \draw[instance guide,dashed] (1.62,0)
                arc[start angle=0,end angle=180,x radius=1.62,y radius=0.42];
            \draw[instance guide] (-1.62,0)
                arc[start angle=180,end angle=360,x radius=1.62,y radius=0.42];
            \draw[instance guide,dashed] (0,0) ellipse[x radius=0.58,y radius=1.62];
            \draw[instance geometry] (0,0) circle (1.62);
            \foreach \angle/\orbitcolor in {18/orbitblue,76/orbitgreen,140/orbitpurple}{
                \draw[instance special,\orbitcolor]
                    (\angle:1.62) to[bend left=14] ({\angle+180}:1.62);
                \fill[\orbitcolor] (\angle:1.62) circle (1.9pt);
                \fill[\orbitcolor] ({\angle+180}:1.62) circle (1.9pt);
            }
            \node[anchor=west,inner sep=3pt] at (18:1.62) {$x$};
            \node[anchor=east,inner sep=3pt] at (198:1.62) {$-x$};
            \fill[white] (0,0) circle (3pt);
            \fill[black!65] (0,0) circle (1.9pt);
            \node[anchor=east,inner sep=3pt] at (0,0) {$O$};
            \node at (0,-2.16) {base: $S^{2}\cong |E_2\Z_2|$};
            \coordinate (forwardstart) at (-0.66,-2.94);
            \coordinate (forwardend) at (0.66,-2.94);
            \node[anchor=east] at (forwardstart) {$x$};
            \node[anchor=west] at (forwardend) {$-x$};
            \draw[instance forward,orbitblue] (forwardstart)
                to[bend left=25] node[above,inner sep=1pt] {$+1$} (forwardend);
            \draw[instance reverse] (forwardend)
                to[bend left=25] node[below,inner sep=1pt] {$-1$} (forwardstart);
        \end{tikzpicture}

    \end{minipage}}\hfill
    \subfigure[Ball $B^4\cong C(S^3)$, with $\Z_3$-orbits]{\begin{minipage}[t]{0.32\linewidth}
        \centering
        \begin{tikzpicture}[x=1.1cm,y=1.1cm,font=\small]
            \path[use as bounding box] (-2.05,-3.35) rectangle (2.05,1.85);
            \coordinate (ballcenter) at (0,1.52);
            \foreach \i/\angle in {0/20,1/140,2/260}{
                \coordinate (a\i) at ({-1.04+0.48*cos(\angle)},{-0.78+0.78*sin(\angle)});
                \coordinate (b\i) at ({1.04+0.48*cos(\angle)},{-0.53+0.78*sin(\angle)});
            }
            \fill[orbitblue!6] (ballcenter) -- (a1) -- (a2) -- (b2) -- cycle;
            \fill[orbitgreen!6] (ballcenter) -- (a0) -- (b0) -- (b1) -- cycle;
            \foreach \i in {0,1,2}{
                \foreach \j in {0,1,2}{
                    \draw[instance guide] (a\i) -- (b\j);
                }
                \draw[instance geometry] (ballcenter) -- (a\i);
                \draw[instance geometry] (ballcenter) -- (b\i);
            }
            \foreach \cx/\cy/\orbitcolor in {
                -1.04/-0.78/orbitblue,
                1.04/-0.53/orbitgreen}{
                \foreach \angle in {20,140,260}{
                    \draw[instance forward,\orbitcolor]
                        ({\cx+0.48*cos(\angle)},{\cy+0.78*sin(\angle)})
                        arc[start angle=\angle,end angle={\angle+120},
                            x radius=0.48,y radius=0.78];
                    \fill[\orbitcolor]
                        ({\cx+0.48*cos(\angle)},{\cy+0.78*sin(\angle)}) circle (1.9pt);
                }
            }
            \node[fill=white,inner sep=1pt] at (-1.04,-0.78) {$S^1$};
            \node[fill=white,inner sep=1pt] at (1.04,-0.53) {$S^1$};
            \node[anchor=west,fill=white,inner sep=1pt,xshift=2pt] at (a0) {$x$};
            \node[anchor=east,inner sep=2pt] at (a1) {$\gamma x$};
            \node[anchor=north west,fill=white,inner sep=1pt,xshift=2pt,yshift=-2pt] at (a2) {$\gamma^2x$};
            \fill[black!65] (ballcenter) circle (1.9pt);
            \node[above,inner sep=3pt] at (ballcenter) {$O$};
            \node at (0,-2.16) {base: $S^3\cong |E_3\Z_2|$};
            \coordinate (forwardstart) at (-0.66,-2.94);
            \coordinate (forwardend) at (0.66,-2.94);
            \node[anchor=east] at (forwardstart) {$x$};
            \node[anchor=west] at (forwardend) {$\gamma x$};
            \draw[instance forward,orbitblue] (forwardstart)
                to[bend left=25] node[above,inner sep=1pt] {$+1$} (forwardend);
            \draw[instance reverse] (forwardend)
                to[bend left=25] node[below,inner sep=1pt] {$-1$} (forwardstart);
        \end{tikzpicture}

    \end{minipage}}\hfill
    \subfigure[Cone $C(|E_d\Z_3|)$, with $\Z_3$-orbits]{\begin{minipage}[t]{0.32\linewidth}
    \centering
    \begin{tikzpicture}[x=0.72cm,y=1.1cm,font=\small]
        \path[use as bounding box] (-3.35,-3.35) rectangle (3.35,1.85);
        \coordinate (apex) at (0,1.52);
        \foreach \i in {0,1,2}{
            \coordinate (u\i) at (-2.15,{-0.48-0.62*\i});
            \coordinate (v\i) at (0.50,{-0.18-0.62*\i});
            \coordinate (omitted\i) at (2.35,{0.12-0.62*\i});
        }
        \fill[orbitblue!7] (apex) -- (u0) -- (v0) -- cycle;
        \fill[orbitgreen!7] (apex) -- (u1) -- (v1) -- cycle;
        \fill[orbitpurple!7] (apex) -- (u2) -- (v2) -- cycle;
        \fill[black!2] (apex) -- (v2) -- (omitted2) -- cycle;
        \foreach \i in {0,1,2}{
            \draw[instance guide,densely dashed] (apex) -- (omitted\i);
            \foreach \j in {0,1,2}{
                \draw[instance guide,densely dashed] (v\i) -- (omitted\j);
            }
        }
        \foreach \i in {0,1,2}{
            \foreach \j in {0,1,2}{
                \draw[instance geometry] (u\i) -- (v\j);
            }
            \draw[instance geometry] (apex) -- (u\i);
            \draw[instance geometry] (apex) -- (v\i);
        }
        \draw[instance forward,orbitblue] (u0) to[bend right=20] (u1);
        \draw[instance forward,orbitblue] (u1) to[bend right=20] (u2);
        \draw[instance forward,orbitblue] (u2)
            .. controls (-3.12,-1.68) and (-3.12,-0.52) .. (u0);
        \draw[instance forward,orbitgreen] (v0) to[bend left=20] (v1);
        \draw[instance forward,orbitgreen] (v1) to[bend left=20] (v2);
        \draw[instance forward,orbitgreen] (v2)
            .. controls (1.35,-1.38) and (1.35,-0.22) .. (v0);
        \node[orbitblue] at (-3.03,-1.10) {$+1$};
        \node[orbitgreen] at (1.52,-0.80) {$+1$};
        \foreach \i/\word in {0/{u},1/{\gamma u},2/{\gamma^2u}}{
            \fill[orbitblue] (u\i) circle (1.9pt);
            \node[anchor=west,fill=white,inner sep=1.5pt,xshift=3pt] at (u\i) {$\word$};
        }
        \foreach \i/\word in {0/{v},1/{\gamma v},2/{\gamma^2v}}{
            \fill[orbitgreen] (v\i) circle (1.9pt);
            \node[anchor=east,fill=white,inner sep=1.5pt,xshift=-3pt] at (v\i) {$\word$};
        }
        \foreach \i in {0,1,2}{
            \node[fill=white,text=black!65,inner sep=1.5pt,font=\large]
                at (omitted\i) {$\cdots$};
        }
        \fill[black!65] (apex) circle (1.9pt);
        \node[above,inner sep=3pt] at (apex) {$O$};
        \node at (0,-2.16) {base: $|E_d\Z_3|$};
        \begin{scope}[x=1.1cm]
            \coordinate (forwardstart) at (-0.66,-2.94);
            \coordinate (forwardend) at (0.66,-2.94);
            \node[anchor=east] at (forwardstart) {$x$};
            \node[anchor=west] at (forwardend) {$\gamma x$};
            \draw[instance forward,orbitblue] (forwardstart)
                to[bend left=25] node[above,inner sep=1pt] {$+1$} (forwardend);
            \draw[instance reverse] (forwardend)
                to[bend left=25] node[below,inner sep=1pt] {$-1$} (forwardstart);
        \end{scope}
    \end{tikzpicture}

    \end{minipage}}
    \endgroup
\caption{An illustration of antipocal vs $\Z_p$ connections, and cones.}
\label{fig:hard-instances}
\end{figure}

To use this obstruction, suppose that we can construct an equivariant map $\psi:|I_{p,k}|\longrightarrow |E_d\Z_p|$. We call the smallest dimension $d$ for which such a map exists the \emph{agreement index}, denoted $\ind_{\Z_p}(|I_{p,k}|)$. This index will determine the dimension needed for our hard metric.

We now describe the hard instance. We use a geometric realization of the cone $C(E_d\Z_p)$, scaled by a factor $r$ with its apex at the origin. Its base is the scaled copy $r|E_d\Z_p|$. Similar to \Cref{sec: cyclic}, we triangulate this cone with $O_{p,k}(1)$ simplex diameter and identify cyclic orbits on its base. If no simplex contains a witness pair, the labels extend to a continuous map from the cone into $|I_{p,k}|$, equivariant on the base. Composing this map with $\psi$ gives precisely the map forbidden by Dold's theorem. Thus, some simplex contains a witness pair, and a similar winding argument forces distortion $\Omega_{p,k,d}(r)$.

The cone has dimension $d+1$, so the resulting metric contains $\Theta_{p,d}(r^{d+1})$ points. Therefore,
\[
    \ind_{\Z_p}(|I_{p,k}|)\le d
    \quad\Longrightarrow\quad
    \text{distortion }
    \Omega_{p,k,d}\!\left(n^{1/(d+1)}\right).
\]
We present the details in \Cref{sec:metric-transfer}.

\paragraph{Bounding the agreement index.}
To obtain the final distortion bound, it remains to bound the agreement index, the smallest $d$ admitting an equivariant map from $|I_{p,k}|$ to $|E_d\Z_p|$. This is a clean notion in algebraic topology that can be explored independently. By combining cyclic symmetry with fractional matching bounds for intersecting hypergraphs~\cite{FG_hypergraph_91} and a characterization of the extremal case~\cite{FUREDI1989217}, we prove $\ind_{\Z_p}(|I_{p,k}|)\le(k-1)^2-1$ for every $k\ge3$ and every prime $p\ge k-1$. Since the cone adds one dimension, this sharpens the warm-up dimension $k(p-1)=O(k^2)$ to $(k-1)^2$, yielding the distortion lower bound $\Omega_k\big(n^{1/(k-1)^2}\big)$. See \Cref{sec:index-upper-bound} for more details.

We also investigate the limits of this approach by proving the complementary lower bound $\ind_{\Z_p}(|I_{p,k}|)=\Omega(k^{\log_3 4})$ for every prime number $p$ and every $k\ge3$. Thus, although there remains room to improve our quadratic index upper bound, the agreement index must be superlinear in $k$. Consequently, this framework cannot yield a desired distortion lower bound with an exponent of order $1/k$. We provide more detail of this part in \Cref{sec:lower_bound} and \Cref{sec:reduction}.

To conclude the overview, we provide a table below comparing the approaches in \cite{CTX26}, our warm-up approach in \Cref{sec: cyclic} using cyclic symmetry, and our real approach using agreement index.

\begin{table}[htbp]
    \centering
    \small
    \renewcommand{\arraystretch}{1.15}
    \renewcommand{\tabularxcolumn}[1]{m{#1}}
    \setlength{\extrarowheight}{5pt}
    \setlength{\tabcolsep}{5pt}

    \begin{tabularx}{\linewidth}{@{}>{\centering\arraybackslash}m{1.15in}|*{3}{>{\centering\arraybackslash}X}@{}}
        \hline
        & \textbf{\cite{CTX26}}
        & \textbf{Cyclic symmetry}
        & \textbf{Our approach} \\
        \hline

        Hard instance
        & ball $B^{d}$ ($=C(E_{d-1}\Z_2)$)
        & ball $B^{d}$ ($=C(|E_{d-1}\Z_2|)$)
        & cone $C(E_{d-1}\Z_p)$ \\
        \hline

        Equivariance on the boundary
        & \makecell[c]{Antipodal\\ ($\Z_2$-equivariance)}
        & $\Z_p$-equivariance
        & $\Z_p$-equivariance \\
        \hline

        Labels
        & $\{-1,+1\}^k$
        & $\Z_p^k$
        & $\Z_p^k$ \\
        \hline

        Topological tool
        & Tucker's lemma
        & An equivariant Borsuk--Ulam theorem (\cite{MV_G_topo})
        & Dold's theorem (\Cref{lem:selfmap}) \\
        \hline

        \makecell[c]{Dimension of\\the hard instance}
        & $2^{k-1}$
        & $k(p-1)$ $(p>k)$
        & $(k-1)^2$ \\
        \hline

        Distortion bound
        & $\Omega_k\left(n^{1/2^{k-1}}\right)$
        & $\Omega_k\left(n^{1/[k(p-1)]}\right)$
        & $\Omega_k\left(n^{1/(k-1)^2}\right)\ (k\geq3)$ \\
        \hline

        Finding witness pair
        & Tucker's lemma
        & Balanced simplex, probability argument
        & Agreement complex, topological obstruction \\
        \hline
    \end{tabularx}

    \caption{Comparison of distortion lower bounds for size-$k$
    tree covers on $n$-point metrics.}
    \label{tab:lower-bound-comparison}
\end{table}

\section{Preliminaries}
\label{sec:preli}

For a positive integer $n$, let $[n] := \{1,\ldots,n\}$. For a real number $x \in \R$, let $(x)_+ := \max\{x, 0\}$. A subscript in an asymptotic notation indicates that the subscripted parameters are treated as constants.

Given an edge-weighted graph $\mathcal{G} = (V(\mathcal{G}),E(\mathcal{G}),w_{\mathcal{G}})$, where $w_{\mathcal{G}} : E(\mathcal{G}) \to \R_{\ge 0}$, and a walk $P = (v_0,v_1,\ldots,v_t)$ in $\mathcal{G}$, we define its length by $\len(P) := \sum_{i=1}^t w_{\mathcal{G}}(\{v_{i-1},v_i\})$. The reverse walk of $P$ is denoted by $P^{-1} := (v_t,v_{t-1},\ldots,v_0)$. Clearly, $\len(P^{-1}) = \len(P)$.

\subsection{Simplices and Triangulations}

A $d$-dimensional simplex is the convex hull $\sigma = \conv\{v_0,\ldots,v_d\}$ of $d+1$ affinely independent points. The set $V(\sigma) := \{v_0,\ldots,v_d\}$ is $\sigma$'s vertex set, and a face of $\sigma$ is the convex hull of a subset of its vertices.

\begin{definition}[Triangulation]
Let $X$ be a subset of $\R^d$. A finite collection $\Gamma$ of simplices is a triangulation of $X$ if $\bigcup_{\sigma \in \Gamma}\sigma = X$, and:
\begin{itemize}
    \item Every face of every $\sigma \in \Gamma$ also belongs to $\Gamma$.
    \item For every $\sigma_1,\sigma_2 \in \Gamma$, the intersection $\sigma_1 \cap \sigma_2$ is either empty or a face of both $\sigma_1$ and $\sigma_2$.
\end{itemize}

We write $V(\Gamma) = \bigcup_{\sigma \in \Gamma}V(\sigma)$ for the vertex set of $\Gamma$, and $\partial \Gamma := \{\sigma \in \Gamma : \sigma \subseteq \partial X\}$ for the boundary of $\Gamma$.
\end{definition}

\subsection{Simplicial Complexes}
We use \emph{abstract simplicial complexes} to abstract the combinatorial structure of the triangulations without fixing a geometric embedding. We give a series of relevant definitions as follows \cite{xie1998endo,meshulam2003domination,matouvsek2003using,kozlov2008combinatorial}.

\begin{definition}[Abstract Simplicial Complex (ASC)]
\label{def:abstract-simplicial-complex}
Let \(V\) be a finite set. A family \(K \subseteq 2^V\) is an \emph{abstract simplicial complex} (ASC) on \(V\) if the following two properties hold:
\begin{itemize}
    \item The empty set \(\varnothing\in K\).
    \item For every \(\sigma \in K\), every subset \(\tau \subseteq \sigma\) also belongs to \(K\).
\end{itemize}
The elements of \(K\) are called its \emph{faces}. A nonempty face \(\sigma\) has \emph{dimension} \(|\sigma|-1\). Faces of dimension \(0\) and \(1\) are called \emph{vertices} and \emph{edges}, respectively. The vertex set of $K$ is denoted by $V(K)$. The dimension of $K$ is defined as $\dim K:=\max\{|\sigma|-1:\varnothing\ne \sigma\in K\}$.
\end{definition}

Let $V$ be a finite set. The \emph{full simplex} on \(V\), denoted by \(\Delta_V\), is the ASC consisting of all subsets of \(V\). The \emph{discrete simplicial complex} on \(V\) is the ASC defined as \(\mathcal{D}(V):=\{\varnothing\}\cup \{\{v\}:v\in V\}.\) For an integer \(d \ge 0\), the \emph{\(d\)-skeleton} of \(K\) is the subcomplex
\[
    K^{(d)} := \{\sigma \in K : |\sigma| \le d+1\},
\]
consisting of the empty face and all faces of dimension at most \(d\) in \(K\). In particular, the \(1\)-skeleton is the graph formed by the vertices and edges of \(K\).

The geometric realization of an ASC turns its compatibility structure into a topological space, enabling the equivariant-map constructions and connectivity arguments used later.

\begin{definition}[Geometric Realization]\label{def:geometric-realization}
    Let $V$ be a finite set and $K$ be an ASC on $V$. For each $v\in V$, let $e_v\in\mathbb R^V$ denote the standard unit vector with a $1$ in coordinate $v$ and zeros elsewhere. The \emph{(geometric) realization} of \(K\) is\footnote{We reserve \(\lvert K\rvert\) for the geometric realization of an ASC \(K\), and write \(\#K\) for its number of faces.}
\[
    |K|=\left\{
        \mu=\sum_{v\in V}\mu_v e_v:
        \mu_v\ge0,\ \sum_{v\in V}\mu_v=1,\
        \supp(\mu)\in K
    \right\},
\]
    where \(\supp(\mu):=\{v:\mu_v>0\}\) is the support of \(\mu\). Each \(\mu\in|K|\) is called a \emph{point} in \(|K|\).\end{definition}

The realization of an ASC is also called a \emph{geometric complex}. A point in \(|K|\) is a probability vector supported on a face in \(K\). The \emph{carrier face} of a point is its support.

We identify each abstract vertex $v$ with $e_v\in\mathbb R^V$, and use the same notation for simplicial maps and actions and their extensions to geometric realizations.

\subsection{Complex Join, Cones and Barycentric Subdivision}

\begin{definition}[Complex Join]\label{def:complex-join}
Let \(V,W\) be disjoint finite sets. Let \(K\) be an ASC on \(V\) and \(L\) be an ASC on \(W\). The \emph{join} of $K$ and $L$ is defined as:
\[
K*L:=\{\sigma\cup\tau: \sigma\in K, \tau\in L\}.
\]
We call the join of copies of the same ASC \(K\) a \(k\)-fold join of \(K\), denoted by \(K^{*(k)}\).
\end{definition}

Given a join of two ASCs $K$ and $L$, every point of its realization can be seen as a weighted combination of two points in $K$ and $L$.

Given a nontrivial finite group $G$, denote \(E_dG=\left(\mathcal{D}(G)\right)^{*(d+1)}\), which is the \((d+1)\)-fold join of the discrete simplicial complexes on \(G\). Its vertices are \((i,g)\), where \(i\) is one of the \(d+1\) layers. A face contains at most one vertex in each layer, so its dimension is at most \(d\). For example, \(E_0G\) is a discrete set, and \(E_1G\) is a complete bipartite graph. $E_n\Z_2$ is homeomorphic to the $n$-dimensional sphere $S^n$. For every point $\mu=\sum_{i=1}^{d+1}\mu_{i}(i,g)\in|E_d G|$, we also denote it as $(\mu_1g_1,\ldots,\mu_{d+1}g_{d+1})$ for simplicity.

\begin{definition}[Cone]
\label{def:cone}
    Given an ASC $K$ on a vertex set $V$ and a new vertex $v$, a \emph{cone} of $K$ and $v$ is defined as
\[
        C(K,v):=|K*\Delta_{\{v\}}|
        \subseteq \mathbb R^{V\cup\{v\}}.
\]
    The vertex $v$ is called the \emph{apex} of the cone.

    In the absence of ambiguity, we write $C(K):=C(K,0)$, where $0$ denotes a new vertex. We will also identify $C(K)$ with a subset of $\mathbb R^V$ by dropping the coordinate corresponding to the apex. Under this identification,
\[
        C(K)=\{tx:0\le t\le1,\ x\in|K|\}.
\]
    In particular, the apex is identified with the origin of $\mathbb R^V$, while the base is identified with $|K|$.
\end{definition}

\begin{definition}[Barycentric Subdivision]
\label{def:barycentric-subdivision}
Let \(V\) be a finite set and \(K\) be an ASC on \(V\). Its \emph{barycentric subdivision} $\sd K$ is an ASC on \(V(\sd K)=K\setminus\{\varnothing\}\) such that,
\[
\sd K=\left\{\{\sigma_0,\ldots,\sigma_q\}:\sigma_0\subsetneq \sigma_1\subsetneq\cdots\subsetneq \sigma_q, \sigma_i\in K \backslash \{\varnothing\}\right\}\cup \{\varnothing\}.
\]
\end{definition}

The barycentric subdivision \(\sd K\) has a vertex for each nonempty face of \(K\), and a simplex for each strict inclusion chain of faces. \(|\sd K|\) is homeomorphic to \(|K|\) for the map \(|\sd K|\to |K|\) that sends a subdivision vertex \(F\) to the barycenter \(\frac{1}{|F|}\sum_{v\in F} v\). The image of the map gives a triangulation of each original simplex in \(K\).

\subsection{Group Actions and Dold's Theorem}

Let $G$ be a finite group with identity element $e$, and let $X$ be a topological space. An action of $G$ on $X$ is a map $(g,x) \mapsto g \cdot x$ satisfying $e \cdot x = x$ and $g \cdot (h \cdot x) = (gh) \cdot x$ for every $g,h \in G$ and $x \in X$. We always assume that the action map is continuous.

The action is \emph{free} if for every $g \in G$ and $x \in X$, $g \cdot x = x$ implies $g = e$. If $G$ acts on both $X$ and $Y$, a map $f \colon X \to Y$ is \emph{$G$-equivariant} if $f(g \cdot x) = g \cdot f(x)$ for every $g \in G$ and $x \in X$.

A map is simplicial if it preserves the faces between ASCs. These maps can be easily extended to the realizations.
\begin{definition}[Simplicial Map]\label{def:simplicial-map}
    Let \(V,W\) be finite sets. Let \(K\) be an ASC on \(V\) and \(L\) be an ASC on \(W\). A \emph{simplicial map} from $K$ to $L$ is determined by a vertex map \(f:V\to W\) such that for every $\sigma\in K$,
\[
    f(\sigma):=\{f(v):v\in\sigma\}\in L.
\]
In the absence of ambiguity, we also write $f: K\to L$. The \emph{affine extension} of \(f\) is defined as a function \(\lvert f\rvert:\lvert K\rvert\to\lvert L\rvert\) such that for every \(\mu\in\lvert K\rvert\),
\[
|f|(\mu)=\sum_{v\in V(K)}\mu_v f(v).
\]
\end{definition}
This gives a continuous map, even when several vertices have the same image.

A simplicial group action permutes the vertices and preserves the faces.

\begin{definition}[Simplicial Action]\label{def:simplicial-action}
    Let $K$ be an ASC on a finite set \(V\), and let $G$ be a finite group that acts on $K$. Then we say that this action is a \emph{simplicial action}, if for every $g\in G$ and $\sigma\in K$, $g\cdot \sigma:=\{g\cdot v: v\in \sigma\}\in K$. This action can be naturally extended to $|K|$.
\end{definition}

We use the following Borsuk--Ulam-type formulation of Dold's theorem~\cite{Dold83}; see also~\cite[Theorem~6.2.5]{matouvsek2003using}.

\begin{theorem}
[Dold's Theorem]
\label{thm:dold}
If \(G\) is a nontrivial finite group and \(d\ge0\), then there is no continuous \(G\)-equivariant map \(|E_{d+1}G|\to |E_dG|.\)
\end{theorem}

Specifically, we use the following equivalent lemma. The proof of the equivalence refers to \Cref{sec:equiv_topo_lemmas}.

\begin{lemma}
\label{lem:selfmap}
Let $G$ be a nontrivial finite group and let $d\ge0$. Let $v$ be a new vertex. There is no continuous map $C(E_dG,v)\to |E_dG|$ whose restriction to the base $|E_dG|$ is $G$-equivariant.
\end{lemma}

\section{Distortion Lower Bounds from the Agreement Index}
\label{sec:metric-transfer}

In this section, we develop a general framework for constructing and analyzing hard instances to establish lower bounds for tree covers using tools from algebraic topology. We ultimately reduce the problem of proving lower bounds for tree covers to an algebraic-topological problem.

The proof follows the same basic ideas as the warm-up in \Cref{sec: cyclic}, but replaces the ball by a cone over $|E_d\mathbb{Z}_p|$ and the balanced-simplex argument by a topological obstruction based on the agreement index. The winding labels and the conversion from label differences to tree distances remain essentially unchanged.

We first define the label space used in this argument.
\begin{definition}[Agreement Complex]
\label{def:agreement-complex}
Let $k\ge 1$, and let $A_1,\ldots,A_k$ be nonempty finite alphabets.  Two words $x,y\in\prod_{i=1}^k A_i$ \emph{agree} if $x_i=y_i$ for some $i\in[k]$.  The \emph{agreement complex}
\[
    I(A_1,\ldots,A_k)
    :=
    \left\{
        \sigma\subseteq \prod_{i=1}^k A_i:
        \forall x,y\in\sigma,\ \exists i\in[k]\text{ such that }x_i=y_i
    \right\}
\]
is the ASC whose faces are the pairwise agreeing word families.  We write $I_{A,k}:=I(A,\ldots,A)$, and in particular $I_{p,k}:=I_{\Z_p,k}$.
\end{definition}

For $G=\Z_p$, let $G$ act diagonally on $I_{p,k}$ by
\[
    g(x_1,\ldots,x_k)=(x_1+g,\ldots,x_k+g),
\]
for every $g \in G$ and vertex $(x_1,\ldots,x_k)$ of $ I_{p,k}$. This defines a simplicial $G$-action on $I_{p,k}$. The induced action on $|I_{p,k}|$ is free: Otherwise, if a nonzero $g$ fixed a point $\mu$, then $g\cdot\supp(\mu)=\supp(\mu)$. Hence for every $x\in\supp(\mu)$, the word $gx$ would also lie in $\supp(\mu)$. But $x$ and $gx$ differ in every coordinate, contradicting the face requirement in \Cref{def:agreement-complex}.

We use $E_dG$ defined in \Cref{def:complex-join} to construct the hard instances. The $G$-action is defined by
\[
g(i,h)=(i,g+h),
\]
for every $g\in G$ and vertex $(i,h)\in E_dG$. This action is also simplicial on $E_dG$ and free on $|E_dG|$.

\begin{definition}[Topological Index \cite{matouvsek2003using}]
\label{def:topological-index}
For a topological space $X$ with a free $G$-action, define
\[
    \ind_G(X)
    :=
    \inf\{d\ge0:\text{there is a continuous $G$-equivariant map }X\to|E_dG|\},
\]
with value $+\infty$ if the set is empty.
\end{definition}

The topological index is a classical conception in equivariant topology \cite{matouvsek2003using}. It has been used to obtain lower bounds on the chromatic numbers of graphs \cite{matousek2002topological,simonyi2013colourful,daneshpajouh2021colorings,daneshpajouh2025box} and uniform hypergraphs \cite{alishahi2016colorful}. We call $\ind_{\Z_p}(|I_{p,k}|)$ the \emph{$(p,k)$-agreement index}.

We can now state the main theorem of this section. It reduces the problem of proving distortion lower bounds for $k$-tree covers to determining the agreement index $\ind_{\Z_p}(|I_{p,k}|)$. In particular, every upper bound on this index immediately yields a corresponding polynomial distortion lower bound through the exponent $1/(\ind_{\Z_p}(|I_{p,k}|)+1)$.

\begin{theorem}[Distortion Lower Bound from the Agreement Index]
\label{thm:transfer}
Let $p$ be a prime and $k\ge2$, and put $d:=\ind_{\Z_p}(|I_{p,k}|)$.  There is a constant $c=c(p,k,d)>0$ such that, for every sufficiently large integer $n$, there is an $n$-point metric $X_n$ for which every $k$-tree cover has distortion at least
\[
    c n^{1/(d+1)}.
\]
\end{theorem}

For the remainder of this section, set $G:=\Z_p$ and $D:=d+1$.  By \Cref{def:topological-index}, fix a $G$-equivariant map $\psi:|I_{p,k}|\to |E_dG|$.  Let $\gamma$ denote the action of $1\in\Z_p$, so $\gamma^a$ denotes the action of $a\in\Z_p$.

The proof of \Cref{thm:transfer} proceeds in four steps. \Cref{sec:hard_instance_construction} constructs a symmetric hard metric from a scaled cone over $E_dG$ and its equivariant triangulation. \Cref{sec:wind_label} assigns winding labels to the trees and establishes their equivariance on the boundary. \Cref{sec:topo_forced_simplex} uses the agreement complex and the topological obstruction to find two vertices of a common simplex whose labels differ in every tree. Finally, \Cref{sec:conclusion} converts these label differences into large tree distances while their original metric distance remains bounded, completing the proof of \Cref{thm:transfer}.

\subsection{Construction of the Hard Instance}
\label{sec:hard_instance_construction}

We construct the hard metric from a triangulated cone over $|E_dG|$.  The triangulation is part of the construction: its vertices form the vertex set before the boundary identifications, and the metric is obtained by identifying the $G$-orbits on the outer boundary.

Set
\[
K:=E_dG=(\mathcal{D}(G))^{*(D)},
\]
and let
\[
C_r:=rC(K)=\{tx:0\le t\le r,\ x\in|K|\}
\]
be the cone from \Cref{def:cone}, scaled by an integer parameter $r\ge1$.  Its apex is $O=0$ and its base is $r|K|$. As described in \Cref{def:cone}, one may also view it as a subspace of $\R^{V(K)}$.

The group $G$ fixes the apex $O$ and acts on $C_r$ by
\[
\gamma^a z
    =
    \sum_{j=0}^{D-1}\lambda_j (j,g_j+a),
\]
for every $a\in\Z$ and $z=\sum_{j=0}^{D-1}\lambda_j (j,g_j)\in C_r$.

We begin by fixing the triangulation whose vertices will be the vertices of the construction.

\begin{restatable}{lemma}{EquivariantTriangulation}
\label{lem:mesh}
For every integer $r\ge1$, the scaled cone $C_r$ has a finite triangulation $\Gamma_r$ such that:
\begin{enumerate}
    \item The base $r|K|$ of $C_r$ is a subcomplex of $\Gamma_r$.
    \item The $G$-action on $\Gamma_r$ is simplicial.

    \item Every simplex has Euclidean diameter at most a constant $M=M(K)$ independent of $r$.

    \item $|V(\Gamma_r)|=\Theta_K(r^D),\ |V(\Gamma_r)\cap r|K||=O_K(r^{D-1}).$
\end{enumerate}
\end{restatable}

The construction starts from the canonical triangulation of $C(K)$ induced by the simplicial structure of $K$, and refines each simplex by a uniform barycentric grid of scale $1/r$, followed by a barycentric triangulation of the resulting cells.  This refinement is compatible on common faces and $G$-equivariant; after scaling, it has constant mesh $M$ and $\Theta_K(r^D)$ vertices.  A formal proof is deferred to \Cref{sec:triangulation}.

Fix such a triangulation $\Gamma_r$ and write $\Lambda_r:=V(\Gamma_r)$ and $\partial\Lambda_r:=\Lambda_r\cap r|K|$. We will use $\Lambda_r$ as the vertex set of the hard instance before the boundary identifications.

We record the two geometric estimates that will be used later.  Define the radial coordinate $\tau:C_r\to[0,r]$ by $\tau(tx)=t$ for every $x \in |K|$.

\begin{lemma}
\label{lem:cone-geometry}
There are constants $\delta=\delta(d)>0$ and $L=L(p,d)>0$ such that, for every $z,z'\in C_r$ and every $a\in G\setminus\{0\}$,
\[
    \|z-\gamma^a z\|_2\ge \delta\,\tau(z),
    \qquad
    |\tau(z)-\tau(z')|\le L\|z-z'\|_2.
\]
One may take $\delta=\sqrt{2/D}$ and $L=\sqrt{pD}$.
\end{lemma}

\begin{proof}
Given $x\in|K|$. Write
\[
    x=\sum_{j=0}^{D-1}\lambda_j (j,g_j),
    \qquad
    \lambda_j\ge 0,
    \qquad
    \sum_{j=0}^{D-1}\lambda_j=1,
\]

For $a\ne 0$,
\[
    \gamma^a x
    =
    \sum_{j=0}^{D-1}\lambda_j (j,g_j+a).
\]
Since $g_j+a\ne g_j$, the two basis vectors $(j,g_j)$ and $(j,g_j+a)$ are distinct and orthogonal.  Hence
\[
    \|x-\gamma^a x\|_2^2
    =
    \sum_{j=0}^{D-1}
    \lambda_j^2
    \bigl\|(j,g_j)-(j,g_j+a)\bigr\|_2^2
    =
    2\sum_{j=0}^{D-1}\lambda_j^2
    \ge
    \frac{2}{D},
\]
where the last inequality follows from $\sum_j\lambda_j=1$.  Therefore, if $z=sx\in C_r$,
\[
    \|z-\gamma^a z\|_2
    =
    s\|x-\gamma^a x\|_2
    \ge
    \sqrt{\frac{2}{D}}\,s
    =
    \sqrt{\frac{2}{D}}\,\tau(z).
\]
Thus we may take $\delta=\sqrt{2/D}$.

For the second estimate, write $z=sx$ as above.  Since the coordinates of $x$ sum to one,
\[
    \sum_{v\in V(K)} z_v
    =
    s\sum_{v\in V(K)}x_v
    =
    s
    =
    \tau(z).
\]
Therefore, for $z,z'\in C_r$,
\[
    |\tau(z)-\tau(z')|
    =
    \left|
        \sum_{v\in V(K)}(z_v-z'_v)
    \right|
    \le
    \sqrt{|V(K)|}
    \left(
        \sum_{v\in V(K)}(z_v-z'_v)^2
    \right)^{1/2}
    =
    \sqrt{pD}\,\|z-z'\|_2.
\]
Thus we may take $L=\sqrt{pD}$.
\end{proof}

We now put the graph structure on the triangulation vertices.  Let $\widetilde{\mathcal{G}}_r$ be the weighted graph with vertex set $\Lambda_r$.  Its \emph{ordinary edges} are the edges of $\Gamma_r$, with weight $w(x,y)=\|x-y\|_2$ for every edge $(x,y)$.  For every boundary vertex $x\in\partial\Lambda_r$, add a zero-weight \emph{special edge} between $x$ and $\gamma x$, with positive orientation $x\to\gamma x$. Thus each boundary orbit $x,\gamma x,\ldots,\gamma^{p-1}x$ is connected by a zero-cost cycle. Contract vertices at zero distance; in particular, each $G$-orbit in $\partial\Lambda_r$ is identified to a single point; no interior vertex is identified.

Denote the quotient graph by $\mathcal{G}_r$.  Let $X_r:=V(\mathcal{G}_r)$, let $d_r$ be the shortest-path metric on $\mathcal{G}_r$, and let $\pi:\Lambda_r\to X_r$ be the quotient map.  After the zero-weight components have been contracted, the shortest-path distance  $d_r$  on $\mathcal{G}_r$ is a metric.

The properties of the triangulation immediately give
\[
    |X_r|=\Theta_{p,d}(r^D),
\]
because only $O_{p,d}(r^{D-1})$ boundary vertices are merged.  Moreover, if $u,v\in\Lambda_r$ lie in one simplex of $\Gamma_r$, then $uv$ is an ordinary edge and
\[
    d_r(\pi(u),\pi(v))\le\|u-v\|_2\le M.
\]
Finally, for $r>M/\delta$, the quotient map $\pi$ is injective on every simplex:  Indeed, two distinct vertices in one simplex could be identified only if they were boundary vertices with $v=\gamma^a u$ for some $a\ne0$; \Cref{lem:cone-geometry} would then give $\|u-v\|_2\ge\delta r>M$, contradicting the mesh bound that $\|u-v\|_2\leq M$.

For the remainder of this section, we assume $r>M/\delta$, so that $\pi$ is injective on every simplex of $\Gamma_r$.  In particular, no ordinary edge has its two endpoints identified by the quotient.

\subsection{Winding Labels}
\label{sec:wind_label}

In this subsection, we define a label for every point in $\Lambda_r=V(\Gamma_r)$. Fix a representative $\widehat A\in\pi^{-1}(A)$ for every $A\in X_r$. These representatives remain fixed throughout; for an interior quotient vertex, the representative is its unique preimage.

If $x,y\in\partial\Lambda_r$ lie in the same $G$-orbit, write $y=\gamma^a x$ with $a\in\{0,\ldots,p-1\}$, and let $S(x,y)$ denote the special-edge walk
\[
    x,\gamma x,\gamma^2x,\ldots,\gamma^a x=y.
\]
For every $x\in\Lambda_r$, let $S(x,x)$ be the trivial walk.

For a directed walk $Q=(A_0,A_1,\ldots,A_m)$ in the quotient graph $\mathcal{G}_r$ with $m\ge1$, choose a corresponding ordinary edge $(x_i^-,x_i^+)$ in $\widetilde{\mathcal{G}}_r$ of the same weight for each traversed edge, with $\pi(x_i^-)=A_{i-1}$ and $\pi(x_i^+)=A_i$. Since $\pi(x_i^+)=\pi(x_{i+1}^-)=A_i$, consecutive lifted edges can be joined by $S(x_i^+,x_{i+1}^-)$. We also insert $S(\widehat A_0,x_1^-)$ at the beginning and $S(x_m^+,\widehat A_m)$ at the end. Finally, we obtain
\[
    \widetilde Q:= S(\widehat A_0,x_1^-)\cup(x_1^-,x_1^+)
    \cup S(x_1^+,x_2^-)\cup(x_2^-,x_2^+)\cup \cdots \cup S(x_{m-1}^+,x_m^-)\cup(x_m^-,x_m^+)
    \cup S(x_m^+,\widehat A_m).
\]
Since every inserted special edge has weight zero, the lifted walk $\widetilde{Q}$ has exactly the same weighted length as $Q$.

Orient each special edge as $x\to\gamma x$ (choosing either orientation when $p=2$). A traversal along this orientation contributes $+1$, while a traversal against it contributes $-1$. For a directed walk $\widetilde Q$ in $\widetilde{\mathcal{G}}_r$, define its winding number by
\[
    \wind(\widetilde Q)
    :=
    \#\{\text{positive special traversals}\}
    -\#\{\text{negative special traversals}\}
    \pmod p.
\]
In particular, a special-edge walk from $x$ to $\gamma^a x$ has winding $a$.

Now fix a dominating tree $T$ on $(X_r,d_r)$ and root it at $o:=\pi(O)$, whose representative is $\widehat o=O$. By decreasing each edge weight $w_T(AB)$ to $d_r(A,B)$, we may assume that every edge in $T$ has weight $d_r(A,B)$. This preserves domination and only reduces tree distances. For every parent--child edge $A\to B$ in $T$, choose a shortest path $Q_T(A,B)$ in $\mathcal{G}_r$ and fix one lift $\widetilde Q_T(A,B)$ from $\widehat A$ to $\widehat B$ as above. For the reverse orientation, set $\widetilde Q_T(B,A):=\widetilde Q_T(A,B)^{-1}$.

For every $A,B\in X_r$, let $\widetilde P_T(A,B)$ be the concatenation of these fixed lifted edge paths along the unique $A$--$B$ path in $T$. Their endpoints match by the choice of representatives, and
\begin{equation}\label{eq:lift-path-length}
    \len\bigl(\widetilde P_T(A,B)\bigr)=d_T(A,B).
\end{equation}
Define $I_T(A):=\wind(\widetilde P_T(o,A))\in G$. Since reverse tree edges use inverse lifted walks and winding is additive under concatenation, for every $A,B\in X_r$,
\begin{equation}\label{eq:winding-to-label}
    \wind\bigl(\widetilde P_T(A,B)\bigr)=I_T(B)-I_T(A).
\end{equation}

For every $x\in\Lambda_r$, set $\theta(x)=0$ if $x$ is interior. Otherwise, let $\theta(x)\in G$ be the unique element such that $x=\gamma^{\theta(x)}\widehat{\pi(x)}$, which exists by freeness on the base. Define
\[
    \ell_T(x):=I_T(\pi(x))+\theta(x)\in G.
\]
For a boundary vertex $x\in\partial\Lambda_r$, we have $\pi(\gamma x)=\pi(x)$ and $\theta(\gamma x)=\theta(x)+1$, and therefore
\begin{equation}\label{eq:boundary-label-equivariance}
    \ell_T(\gamma x)=\ell_T(x)+1.
\end{equation}

For every $u,v\in\Lambda_r$, put $A:=\pi(u)$ and $B:=\pi(v)$. Define $\widetilde P_T(u,v)$ by concatenating $S(u,\widehat A)$, $\widetilde P_T(A,B)$, and $S(\widehat B,v)$ in this order. The endpoint corrections have length zero and winding $-\theta(u)$ and $\theta(v)$, respectively; they are trivial for interior endpoints. By \eqref{eq:lift-path-length} and \eqref{eq:winding-to-label},
\begin{equation}\label{eq:lift-properties}
    \len\bigl(\widetilde P_T(u,v)\bigr)=d_T(\pi(u),\pi(v)),
    \qquad
    \wind\bigl(\widetilde P_T(u,v)\bigr)=\ell_T(v)-\ell_T(u).
\end{equation}

Given $k$ dominating trees $T_1,\ldots,T_k$, define the winding label of $x$ as
\[
    \ell(x):=(\ell_{T_1}(x),\ldots,\ell_{T_k}(x))\in G^k.
\]
By \eqref{eq:boundary-label-equivariance}, on the boundary we have
\begin{equation}\label{eq:joint_equivariance}
    \ell(\gamma x)=\ell(x)+(1,\ldots,1).
\end{equation}

\subsection{A Topologically Forced Simplex}
\label{sec:topo_forced_simplex}

Recall that, by \Cref{def:topological-index}, we fixed a $G$-equivariant map
\[
\psi:|I_{p,k}|\longrightarrow |E_dG|.
\]
Given $k$ trees, their labels form words in $G^k$. If no pair in a simplex changes every coordinate, then all label words in that simplex pairwise agree and therefore form a face of $I_{p,k}$. Composing the resulting simplicial map with $\psi$ then produces a continuous map from $C(|E_dG|)$ to $|E_dG|$ whose restriction to the base is $G$-equivariant, which is impossible by \Cref{lem:selfmap}.

\begin{lemma}
\label{lem:all-change}
Let $T_1,\ldots,T_k$ be arbitrary dominating trees on $(X_r,d_r)$. For every sufficiently large $r$, some simplex $\sigma\in\Gamma_r$ contains vertices $u,v$ such that for every $j\in[k]$,
\[
    \ell_{T_j}(u)\ne\ell_{T_j}(v).
\]
\end{lemma}

\begin{proof}
Suppose no such simplex exists. For any two vertices $u,v$ of a simplex, the assumption implies that their label words agree in at least one coordinate. Hence the image of every simplex is a pairwise agreeing family. By \Cref{lem:mesh}, the $G$-action on $\Gamma_r$ is simplicial. Therefore the vertex map $\ell : V(\Gamma_r) \to G^k$ is a simplicial map
\[
    \ell:\Gamma_r\longrightarrow I_{p,k}.
\]

Since $\Gamma_r$ is a triangulation on $C_r$, we further extend $\ell$ affinely over every simplex of $\Gamma_r$. More explicitly, for every point $z\in C_r$, assume $z$ lies in the simplex $\sigma\in \Gamma_r$ and $z=\sum_{x \in V(\sigma)}\lambda_x x$, we set
\[
|\ell|(z)=\sum_{x \in V(\sigma)}\lambda_x \ell(x).
\]
Thus $|\ell|$ is a continuous map $C_r\to |I_{p,k}|$.

By \Cref{lem:mesh}, the base $r|E_dG|$ is a subcomplex of $\Gamma_r$. Hence the vertex-level equivariance in \eqref{eq:joint_equivariance} extends affinely over every simplex of the base. More explicitly, for every point $z\in r |E_d G|$, assume $z$ lies in the simplex $\sigma\in \Gamma_r$ and $z=\sum_{x \in V(\sigma)}\lambda_x x$, then
\[
|\ell|(\gamma z)=\sum_{x \in V(\sigma)}\lambda_x \ell(\gamma x)
=\gamma\sum_{x \in V(\sigma)}\lambda_x \ell(x)
=\gamma|\ell|(z).
\]
Thus $|\ell|$ is $G$-equivariant on the base $r |E_d G|$.

Compose $|\ell|$ with the fixed continuous $G$-equivariant map $\psi:|I_{p,k}|\to|E_dG|$ to obtain
\[
    F:=\psi\circ|\ell|:C_r\longrightarrow |E_dG|.
\]
Then $F$ is continuous, and its restriction to the base $r|E_dG|$ is $G$-equivariant.

Let $\phi_r(x)=rx$ be the scaling map from $C(E_dG)$ to $C_r$. Since the $G$-action is linear, $\phi_r$ is continuous and $G$-equivariant. Therefore
\[
    \widetilde F:=F\circ\phi_r:
    C(E_dG)\longrightarrow |E_dG|
\]
is continuous, and its restriction to the base $|E_dG|$ is $G$-equivariant. This contradicts \Cref{lem:selfmap}.
\end{proof}

\subsection{From Labels to Tree Distances}
\label{sec:conclusion}

Fix the simplex $\sigma$ and vertices $u,v$ supplied by \Cref{lem:all-change}.  Because $u,v$ lie in one simplex, their distance in the hard metric is at most $M = O(1)$.  We now show that a nonzero change in one tree label forces that tree to travel distance $\Omega(r)$.

\begin{lemma}
\label{lem:unwinding}
If a walk $P$ in $\widetilde{\mathcal{G}}_r$ goes from $u$ to $v$ and has winding $a\in G$, then there is an ordinary-edge-only walk from $u$ to $\gamma^{-a}v$ with the same weighted length.
\end{lemma}

\begin{proof}
Decompose $P$ into ordinary subwalks $P_0,P_1,\ldots,P_s$ separated by its successive special traversals.  For $j=0,\ldots,s$, let $a_j\in G$ be the cumulative winding before $P_j$, so that $a_0=0$ and $a_s=a$.

Translate every vertex and edge of $P_j$ by $\gamma^{-a_j}$.  These translated ordinary subwalks concatenate.  Indeed, if the $j$-th special traversal goes from $x$ to $\gamma^{\varepsilon_j}x$, where $\varepsilon_j\in\{1,-1\}$, then $a_j=a_{j-1}+\varepsilon_j$, and $\gamma^{-a_j}\gamma^{\varepsilon_j}x=\gamma^{-a_{j-1}}x$. Thus the translated copy of $P_{j-1}$ ends exactly where the translated copy of $P_j$ begins.

Their concatenation is an ordinary-edge-only walk starting at $u$ and ending at $\gamma^{-a}v$.  The $G$-action preserves ordinary edges and their weights, while the removed special edges have weight zero, so the weighted length is unchanged.
\end{proof}

\begin{lemma}
\label{lem:radial}
There are constants $c_0=c_0(p,d)>0$ and $r_0=r_0(p,d)$ such that, whenever $r\ge r_0$, $u,v$ lie in one simplex of $\Gamma_r$, and $\ell_T(u)\ne\ell_T(v)$,
\[
    d_T(\pi(u),\pi(v))\ge c_0r.
\]
\end{lemma}

\begin{proof}
Let $a:=\ell_T(v)-\ell_T(u)\ne0$.  By \eqref{eq:lift-properties} and \Cref{lem:unwinding}, there is an ordinary walk from $u$ to $\gamma^{-a}v$ whose length is $d_T(\pi(u),\pi(v))$.  Since ordinary edge weights are Euclidean lengths,
\[
    d_T(\pi(u),\pi(v))\ge\|u-\gamma^{-a}v\|_2.
\]

Suppose first that $\max\{\tau(u),\tau(v)\}\ge r/2$, and by symmetry assume $\tau(v)\ge r/2$.  Using \Cref{lem:cone-geometry} and the fact that $u,v$ lie in one simplex,
\[
    d_T(\pi(u),\pi(v))
    \ge \|v-\gamma^{-a}v\|_2-\|u-v\|_2
    \ge \delta\tau(v)-M
    \ge \frac{\delta r}{2}-M.
\]

Now suppose $\tau(u),\tau(v)<r/2$.  The walk in \eqref{eq:lift-properties} has nonzero winding, hence uses a special edge, and every special edge lies at radial level $r$.  Before its first special traversal the walk must travel from radial level $\tau(u)$ to $r$, and after its last special traversal it must return from $r$ to $\tau(v)$.  The second bound in \Cref{lem:cone-geometry}, applied to these two disjoint ordinary portions, gives
\[
    d_T(\pi(u),\pi(v))
    \ge \frac{r-\tau(u)}{L}+\frac{r-\tau(v)}{L}
    >\frac rL.
\]
Taking $c_0:=\min\{\delta/4,1/L\}$ and $r_0\ge4M/\delta$ proves the claim.
\end{proof}

We can now finish the proof of the tree cover lower bound.

\begin{proof}[Proof of \Cref{thm:transfer}]
    Let $T_1,\ldots,T_k$ be a tree cover of $X_r$, and take $u,v$ from \Cref{lem:all-change}.  For sufficiently large $r$, the quotient map is injective on their common simplex, so $\pi(u)\ne\pi(v)$.  By the construction of $d_r$ and \Cref{lem:radial}, for every $j\in[k]$,
\[
    d_r(\pi(u),\pi(v))\le M,
    \qquad
    d_{T_j}(\pi(u),\pi(v))\ge c_0r.
\]
Thus every $k$-tree cover of $X_r$ has stretch at least
\[
\beta_r:= \frac{c_0}{M} r = \Omega_{p,k,d}(r).
\]
Since $N_r:=|X_r|=\Theta_{p,d}(r^D)$, this gives $\Omega_{p,k,d}(N_r^{1/D})$ along the family $(X_r)$.

It remains to obtain exactly $n$ terminals.  Write $N_r:=|X_r|$, and for a sufficiently large $n$ choose the largest integer $r$ such that $N_r\le n$.  Add the remaining $n-N_r$ terminals so that the distance between every two distinct points, at least one of which is a new point, is a common value
\[
    Q>\beta_r \cdot \diam(X_r)
    :=\beta_r \cdot \max_{x,y\in X_r}d_r(x,y).
\]
This defines a metric.

Suppose that the enlarged metric admitted a size-$k$ tree cover of stretch $\alpha<\beta_r$.  For every pair of old terminals $x,y\in X_r$, choose a tree certifying this pair.  Its $x$--$y$ path cannot contain a new terminal $z$: otherwise, by domination,
\[
    d_T(x,y)
    =
    d_T(x,z)+d_T(z,y)
    \ge 2Q
    >
    \beta_r \cdot \diam(X_r),
\]
which is  a contradiction.

Thus every old pair has a certifying path using only old terminals.  Removing all new terminals from the trees leaves forests on $X_r$ that obtain these certifying paths.  Reconnect the components of each forest arbitrarily, assigning the new edges sufficiently large weights so that the resulting spanning trees dominate $d_r$.  Paths inside the original components are unchanged, so these trees would form a size-$k$ tree cover of $X_r$ with stretch less than $\beta_r$, which is a contradiction.

Finally, since $N_r=\Theta_{p,d}(r^D)$, the maximal choice of $r$ satisfies $r=\Omega_{p,d}(n^{1/D})$.  Therefore the resulting $n$-point metric has tree-cover stretch
\[
    \Omega_{p,k,d}\!\left(n^{1/D}\right)
    =
    \Omega_{p,k,d}\!\left(n^{1/(d+1)}\right).
\]
This proves \Cref{thm:transfer}.
\end{proof}

\section{Upper Bounds of the Agreement Index}\label{sec:index-upper-bound}

We now upper-bound the index of the agreement complex. For a point $\mu\in|I_{p,k}|$, its marginals record how the mass of $\mu$ is distributed among the symbols in each coordinate. The key observation is that some marginal atom must have mass at least $1/(k-1)$. If this inequality is strict, the excess mass is supported on at most $k-2$ symbols and yields a low-dimensional equivariant map. The only remaining case is equality in every coordinate. F{\"u}redi's extremal characterization \cite{FUREDI1989217} shows that the support of such a point has a rigid finite-geometric structure, which can be incorporated using one additional copy of $G$. Combining the two cases gives the desired quadratic upper bound on the index.

Throughout this section, let $k\ge3$, let $p\ge k-1$ be prime, put $G:=\Z_p$, and write $X:=|I_{p,k}|$. For $\mu\in X$, define its weight on $x\in V(I_{p,k})=G^k$ as $\mu(x)$ and define its $i$-th marginal by
\[
    \mu_i(a):=\sum_{x\in G^k:x_i=a}\mu(x),
\]
for every $a\in G$, and let
\[
    M_i(\mu):=\max_{a\in G}\mu_i(a)
    \qquad\text{and}\qquad
    M(\mu):=\max_{i\in[k]}M_i(\mu).
\]

\begin{theorem}
\label{thm:index-upper-bound}
For every $k\ge3$ and every prime $p\ge k-1$,
\[
    \ind_G(X)\le (k-1)^2-1.
\]
\end{theorem}

Combining this index upper bound with \Cref{thm:transfer} immediately yields the corresponding tree-cover lower bound. In particular, choosing a prime $p\ge k-1$ allows us to take $d+1\le (k-1)^2$ in the metric construction.

\begin{theorem}[Distortion Lower Bound]
\label{thm:quadratic-tree-cover-lower-bound}
For every $k\ge3$ and every sufficiently large integer $n$, there exists an $n$-point metric for which every $k$-tree cover has distortion
\[
    \Omega_k\!\left(n^{1/(k-1)^2}\right).
\]
\end{theorem}

\begin{proof}
Let $p$ be the smallest prime with $p\ge k-1$, and put $d:=\ind_{\Z_p}(|I_{p,k}|)$.  According to Bertrand's postulate~\cite{bertrand1845}, $p<2k$. And by \Cref{thm:index-upper-bound}, $d\le (k-1)^2-1$.  Applying \Cref{thm:transfer} therefore gives a $k$-tree-cover distortion lower bound of
\[
    \Omega_k\!\left(n^{1/(d+1)}\right)
    \ge
    \Omega_k\!\left(n^{1/(k-1)^2}\right).
\]
\end{proof}

\subsection{Fractional Matchings and Large Marginals}

An \emph{$r$-partite $r$-uniform hypergraph} consists of disjoint vertex classes $V_1,\ldots,V_r$ and a finite family of edges, each containing exactly one vertex from each class. It is \emph{intersecting} if every two edges have a common vertex.

For our application, a word $x=(x_1,\ldots,x_k)\in G^k$ is identified with the edge
\[
    e_x:=\{(1,x_1),\ldots,(k,x_k)\}
\]
in the $k$-partite $k$-uniform hypergraph whose $i$-th vertex class is $\{i\}\times G$. For $\mu\in X$, let $\mathcal H_\mu$ be the hypergraph with edge set $\{e_x:x\in\supp(\mu)\}$. Since $\supp(\mu)$ is pairwise agreeing, $\mathcal H_\mu$ is intersecting. Isolated vertices play no role and will be ignored when convenient.

A \emph{fractional matching} on a hypergraph assigns a weight $f_e\ge0$ to each edge so that
\[
    \sum_{e\ni v}f_e\le1
\]
for every vertex $v$. Its value is $\sum_e f_e$, and the maximum possible value is denoted by $\nu^*$.

For $\mathcal H_\mu$, if the edge $e_x$ is given weight $\mu(x)$, then the load at $(i,a)$ is exactly
\[
    \sum_{x:x_i=a}\mu(x)=\mu_i(a).
\]
Thus the marginal masses are precisely the vertex loads of this weighted hypergraph.

\begin{theorem}[{\cite[Theorem 2]{FG_hypergraph_91}}]
\label{thm:fractional-ryser}
Let $r \ge 2$. If $\mathcal H$ is an intersecting $r$-partite $r$-uniform hypergraph, then
\[
    \nu^*(\mathcal H)\le r-1.
\]
\end{theorem}

\begin{corollary}
\label{cor:large-marginal}
Every $\mu\in X$ satisfies
\[
    M(\mu)\ge\frac1{k-1}.
\]
\end{corollary}

\begin{proof}
If $M(\mu)<1/(k-1)$, assign to the edge $e_x$ the weight $f_x:=\mu(x)/M(\mu)$. The load at $(i,a)$ is $\mu_i(a)/M(\mu)\le1$, so $f$ is a fractional matching. Its value is
\[
    \sum_{x \in G^k} f_x=\sum_{i=1}^k\sum_{a\in G}\frac{\mu_i(a)}{M(\mu)}=\sum_{x\in G^k}\frac{\mu(x)}{M(\mu)}=\frac1{M(\mu)}>k-1,
\]
contradicting \Cref{thm:fractional-ryser} with $r=k$.
\end{proof}

\subsection{The Strict Case}

This subsection handles the points for which some marginal atom has mass strictly larger than $1/(k-1)$. For such a marginal, the mass above the threshold can be supported on at most $k-2$ symbols, and hence determines a point in the $(k-3)$-skeleton of $\Delta_G$. We map this skeleton equivariantly to $E_{k-3}G$, obtaining one local map for each coordinate, and then combine these $k$ maps to obtain a single equivariant map on the whole strict region.

We use the following standard property.

\begin{lemma}[{\cite[Proposition 6.2.4(v)]{matouvsek2003using}}]
\label{lem:free-complex-to-join}
Let $H$ be a nontrivial finite group, and let $K$ be a nonempty finite ASC with a simplicial $H$-action that is free on $|K|$. If $\dim K\le d$, then there is a continuous $H$-equivariant map $|K|\to |E_dH|$.
\end{lemma}

We let $G$ act simplicially on $\Delta_G$ by translation on its vertex set; on $|\Delta_G|\subseteq\R^G$, this action permutes coordinates according to
\[
(s\cdot a)_g=a_{g-s}.
\]
The skeleton $\Delta_G^{(k-3)}$ is a $G$-invariant subcomplex.

\begin{lemma}
\label{lem:high-marginal-compression}
Let
\[
    U:=\left\{a\in|\Delta_G|:\max_{g\in G}a_g>\frac1{k-1}\right\}.
\]
Then \(U\) is a \(G\)-invariant open subspace of \(|\Delta_G|\), and there is a continuous \(G\)-equivariant map \(U\to|E_{k-3}G|\).
\end{lemma}

\begin{proof}
For $a\in U$, define
\[
    b_g(a)
    :=
    \frac{(a_g-1/(k-1))_+}
         {\sum_{h\in G}(a_h-1/(k-1))_+}.
\]
The denominator is positive on $U$, so $b$ is continuous. Moreover, both the numerator and the denominator commute with the coordinate permutation induced by translation. Hence $b(s\cdot a)=s\cdot b(a)$ for every $s\in G$, i.e. $b$ is $G$-equivariant. Additionally, at most $k-2$ coordinates of $b(a)$ are nonzero, since $k-1$ coordinates strictly larger than $1/(k-1)$ cannot have total mass one. Thus $b(a)$ lies in $|\Delta_G^{(k-3)}|$, and $b$ is a continuous $G$-equivariant map from $U$ to $|\Delta_G^{(k-3)}|$.

The induced $G$-action on $|\Delta_G^{(k-3)}|$ is free. Indeed, if a nonzero $s\in G=\Z_p$ fixed a point, it would preserve its nonempty support. Every nonzero translation of $G$ has orbits of size $p$, while such a support has at most $k-2<p$ vertices, a contradiction. By \Cref{lem:free-complex-to-join}, $|\Delta_G^{(k-3)}|$ admits a continuous $G$-equivariant map to $|E_{k-3}G|$. Composing the two continuous $G$-equivariant maps proves the lemma.
\end{proof}

For each $i\in[k]$, let $m_i:X\to|\Delta_G|$ be the marginal map $m_i(\mu):=\mu_i$, and set
\[
    U_i:=m_i^{-1}(U)=\left\{\mu\in X:M_i(\mu)>\frac1{k-1}\right\}.
\]
The map $m_i$ is continuous, since each $\mu_i(a)$ is a finite linear sum of the coordinates of $\mu$. With the diagonal action on $X$ and the translation action on $|\Delta_G|$, it is also $G$-equivariant: translating every word by $s\in G$ translates the marginal by the same element. Since $U$ is $G$-invariant and open, $U_i$ is a $G$-invariant open subspace of $X$. Restricting $m_i$ to $U_i$ and composing with \Cref{lem:high-marginal-compression} therefore gives a continuous $G$-equivariant map
\[
    f_i:U_i\longrightarrow |E_{k-3}G|.
\]
On the strict region
\[
    X_{>}:=\left\{\mu\in X:M(\mu)>\frac1{k-1}\right\}
          =\bigcup_{i=1}^k U_i,
\]
which is a $G$-invariant open subspace of $X$, define
\[
    w_i(\mu):=\left(M_i(\mu)-\frac1{k-1}\right)_+,
    \qquad
    W_{>}(\mu):=\sum_{i=1}^k w_i(\mu).
\]
Each $M_i$ is continuous, being the maximum of finitely many continuous marginal coordinates, and is $G$-invariant because translation merely permutes those coordinates. Hence each $w_i$ is continuous and $G$-invariant, with $\{\mu : w_i(\mu)>0\}=U_i$.

Since $W_{>}>0$ on $X_{>}$, define $\lambda_i:=w_i/W_{>}$. We equip $(E_{k-3}G)^{*(k)}$ with the diagonal $G$-action induced from the standard action on each factor. Using the block-coordinate realization of the join, define
\[
    F_{>}(\mu)
    :=
    \bigl(
        \lambda_1(\mu)f_1(\mu),
        \ldots,
        \lambda_k(\mu)f_k(\mu)
    \bigr)
    \in
    |(E_{k-3}G)^{*(k)}|.
\]
When $\lambda_i(\mu)=0$, the $i$-th factor is irrelevant, so $f_i(\mu)$ need not be evaluated outside $U_i$. Each $\lambda_i$ is continuous and $G$-invariant. On $U_i$, the $i$-th block is continuous; as $\mu$ approaches a point with $\lambda_i(\mu)=0$, the block converges to $0$ because $f_i(\mu)$ remains in the fixed compact space $|E_{k-3}G|$. Thus $F_{>}$ is continuous. Finally, the $G$-invariance of the $\lambda_i$ and the $G$-equivariance of the $f_i$ give $F_{>}(g\mu)=gF_{>}(\mu)$, so $F_{>}$ is $G$-equivariant.

Since
\[
    (E_{k-3}G)^{*(k)}
    =
    (\mathcal D(G))^{*(k(k-2))}
    =
    E_{k(k-2)-1}G,
\]
and this natural identification preserves the $G$-action on the join layers, the strict region already admits an equivariant map to $|E_{k(k-2)-1}G|$. The equality locus is the only remaining obstruction, and the next subsection shows that it can be absorbed by one additional copy of $G$.

\subsection{The Equality Case}

We now handle the points excluded from the strict region. The first step is to identify exactly what equality can look like; we then use the supporting face of such a point to recover one additional $G$-coordinate.

Suppose that $\mu\in X$ does not belong to the strict region. By \Cref{cor:large-marginal},
\[
    M_i(\mu)\le\frac1{k-1}
    \quad(i\in[k]),
    \qquad
    M(\mu)=\frac1{k-1}.
\]
Then the weights $f_x:=(k-1)\mu(x)$ form a fractional matching of value $k-1$ on $\mathcal H_\mu$. Hence
\[
    \nu^*(\mathcal H_\mu)=k-1.
\]

The only additional hypergraph input we need is F{\"u}redi's characterization of the equality case. We first introduce the special form in which that structure appears here.

\begin{definition}[Affine-Plane Configuration]
\label{def:affine-plane-code}
For a set $C\subseteq G^k$ and $i\in[k]$, write $A_i(C):=\{x_i:x\in C\}$. We call $C$ an \emph{affine-plane configuration} if
\[
    |C|=(k-1)^2,
    \qquad
    |A_i(C)|=k-1\quad(i\in[k]),
\]
and every two distinct words $x,y\in C$ satisfy $|\{i\in[k]:x_i=y_i\}|=1$.
\end{definition}

\begin{theorem}[F{\"u}redi Equality Characterization {\cite[Theorem 2.1]{FUREDI1989217}}]
\label{thm:furedi-equality}
Let $S\subseteq G^k$ be a pairwise agreeing family of words, and let $\mathcal H_S$ be the associated $k$-partite $k$-uniform hypergraph with edges
\[
    e_x:=\{(1,x_1),\ldots,(k,x_k)\}
    \quad(x\in S).
\]
If
\[
    \nu^*(\mathcal H_S)=k-1,
\]
then $S$ is an affine-plane configuration.
\end{theorem}

The theorem identifies the possible supports. For our application we also need to know the point $\mu$ itself, not only its support; the next proposition shows that equality forces the unique uniform measure on such a configuration.

\begin{proposition}
\label{prop:equality-classification}
A point $\mu\in X$ satisfies
\[
    M_i(\mu)\le\frac1{k-1}
    \quad(i\in[k])
\]
if and only if $\supp(\mu)$ is an affine-plane configuration and $\mu$ is uniform on its support.
\end{proposition}

\begin{proof}
Let $\mu\in X$ satisfy $M_i(\mu)\le1/(k-1)$ for every $i$. As observed above, the weights $f_x:=(k-1)\mu(x)$ form a fractional matching on $\mathcal H_\mu$ of value $k-1$. Hence \Cref{thm:furedi-equality} implies that $\supp(\mu)$ is an affine-plane configuration $C$. It remains to determine the edge weights.

For a vertex $v$ of $\mathcal H_\mu$, define its $f$-load by
\[
    \mathrm{load}_f(v):=
    \sum_{y\in C:\,v\in e_y}f_y.
\]
Fix a coordinate $i$. Every edge contains exactly one vertex in the $i$-th part, so
\[
    \sum_{v\in\{i\}\times G}\mathrm{load}_f(v)
    =
    \sum_{y\in C}f_y
    =
    k-1.
\]
Exactly $k-1$ vertices are used in this part, and every vertex load is at most one because $f$ is a fractional matching. Hence every used vertex has load exactly one.

Now fix $x\in C$. Since $e_x$ contains one used vertex from each part,
\[
    k
    =
    \sum_{v\in e_x}\mathrm{load}_f(v)
    =
    \sum_{y\in C}|e_x\cap e_y|f_y.
\]
For $y\ne x$, the words $x$ and $y$ agree in exactly one coordinate, so $|e_x\cap e_y|=1$, while $|e_x\cap e_x|=k$. Therefore
\[
    k
    =
    kf_x+\sum_{y\ne x}f_y
    =
    (k-1)+(k-1)f_x.
\]
Thus $f_x=1/(k-1)$ for every $x\in C$. Since $f_x=(k-1)\mu(x)$, we obtain $\mu(x)=1/(k-1)^2$ for every $x\in C$, so $\mu$ is uniform on its support.

Conversely, suppose that $\supp(\mu)=C$ is an affine-plane configuration and $\mu$ is uniform on $C$. Fix a coordinate $i$ and a symbol $a\in A_i(C)$, and let $C_{i,a}:=\{x\in C:x_i=a\}$. For any $j\ne i$, two distinct words in $C_{i,a}$ must have different $j$-th coordinates, since they already agree in coordinate $i$. Hence $|C_{i,a}|\le k-1$. The $k-1$ classes $C_{i,a}$ partition $C$, which has $(k-1)^2$ elements, so every class has size exactly $k-1$. Since $\mu$ is uniform on $C$, every used marginal atom has mass $1/(k-1)$. Thus $M_i(\mu)=1/(k-1)$ for every $i$.
\end{proof}

To extend the map continuously, we will use not only the uniform equality point but the relative interior of its supporting face. We therefore record that these supporting faces are already maximal in the agreement complex.

\begin{lemma}
\label{lem:affine-plane-code-maximality}
Every affine-plane configuration is a maximal face of $I_{p,k}$.
\end{lemma}

\begin{proof}
Let $C$ be an affine-plane configuration. By definition, every two distinct words in $C$ agree in a coordinate, so $C$ is a face of $I_{p,k}$. Let $\mu_C$ be the uniform measure on $C$. By the converse part of \Cref{prop:equality-classification}, assigning weight $1/(k-1)$ to every edge of $\mathcal H_C$ gives a fractional matching of value $k-1$, and hence $\nu^*(\mathcal H_C)=k-1$.

Now suppose that an agreement face $S$ contains $C$. Then $\mathcal H_S$ is intersecting and contains $\mathcal H_C$, so $\nu^*(\mathcal H_S)\ge k-1$. Together with \Cref{thm:fractional-ryser}, equality holds. By \Cref{thm:furedi-equality}, $S$ is itself an affine-plane configuration and therefore has exactly $(k-1)^2$ elements. Since $C\subseteq S$ already has $(k-1)^2$ elements, we must have $S=C$.
\end{proof}

It remains to extend the map over the equality locus. By \Cref{prop:equality-classification,lem:affine-plane-code-maximality}, every equality point lies in the relative interior of a maximal affine-plane-configuration face. We use the phase of this face to define the one additional $G$-coordinate.

Let $\mathcal C$ be the finite set of all affine-plane configurations in $G^k$. The diagonal action of $G$ on $I_{p,k}$ induces an action on its faces by $C\mapsto gC:=\{gx:x\in C\}$. Translation preserves all three defining properties of an affine-plane configuration, so $\mathcal C$ is $G$-invariant. The restricted action on $\mathcal C$ is free: if a nonzero $g\in G$ fixed some $C\in\mathcal C$ setwise, then it would fix the uniform measure on the face $C$, contradicting the freeness of the $G$-action on $X$.

Choose one representative $C_0$ from each $G$-orbit in $\mathcal C$. Every $C\in\mathcal C$ can be written uniquely as $C=gC_0$; call $g$ its \emph{phase}. On
\[
    U_0:=\{\mu\in X:\supp(\mu)\in\mathcal C\},
\]
define $f_0(\mu)$ to be the phase of $\supp(\mu)$. The set $U_0$ is the disjoint union of the relative interiors of the faces in $\mathcal C$, and these relative interiors are open in the subspace $U_0$ because their proper faces are omitted. Hence $f_0$ is continuous, being constant on each such component. Moreover, if $\supp(\mu)=gC_0$, then $\supp(h\mu)=hgC_0$, so $f_0(h\mu)=hf_0(\mu)$. Thus
\[
    f_0:U_0\longrightarrow |\mathcal{D}(G)|
\]
is a continuous $G$-equivariant map.

We next need a continuous $G$-invariant weight that is positive exactly where the phase map is defined. Define
\[
    w_0(\mu):=
    \sum_{C\in\mathcal C}\prod_{x\in C}\mu(x),
\]
and obtain
\[
    w_i(\mu):=
    \left(M_i(\mu)-\frac1{k-1}\right)_+
    \quad(i\in[k]).
\]
By \Cref{lem:affine-plane-code-maximality}, $w_0(\mu)>0$ if and only if $\mu\in U_0$: positivity of one product means that some $C\in\mathcal C$ is contained in $\supp(\mu)$, and maximality then forces $\supp(\mu)=C$. The function $w_0$ is continuous, being a finite sum of monomials in the coordinates of $\mu$. It is also $G$-invariant: the action of $G$ permutes the configurations in $\mathcal C$ and simultaneously permutes the coordinates of $\mu$, so the summands in the definition of $w_0$ are permuted. By \Cref{cor:large-marginal,prop:equality-classification},
\[
    W(\mu):=
    w_0(\mu)+\sum_{i=1}^k w_i(\mu)>0
\]
for every $\mu\in X$.

Set $\lambda_i:=w_i/W$ for $i=0,\ldots,k$, and define
\[
    F(\mu)
    :=
    \bigl(
        \lambda_0(\mu)f_0(\mu),
        \lambda_1(\mu)f_1(\mu),
        \ldots,
        \lambda_k(\mu)f_k(\mu)
    \bigr)
    \in
    |\mathcal D(G)*(E_{k-3}G)^{*(k)}|.
\]
As before, a factor with zero coefficient is irrelevant and need not be evaluated outside its domain. The same argument used for $F_{>}$ shows that $F$ is continuous: the only new block is $\lambda_0f_0$, where $f_0$ is continuous on $U_0=\{w_0>0\}$ and the block tends to $0$ as $\lambda_0\to0$. Since all $\lambda_i$ are $G$-invariant and all $f_i$, including $f_0$, are $G$-equivariant on their domains, $F$ is also $G$-equivariant.

Finally,
\[
    \mathcal D(G)*(E_{k-3}G)^{*(k)}
    =
    (\mathcal D(G))^{*(1+k(k-2))}
    =
    E_{(k-1)^2-1}G.
\]
Therefore
\[
    \ind_G(X)\le(k-1)^2-1.
\]
This proves \Cref{thm:index-upper-bound}.

\section{Lower Bounds of the Agreement Index}
\label{sec:lower_bound}

\begin{theorem}
\label{thm:lower_bound}
Let $p$ be a prime and let $k\ge3$ be an integer.  Define $m=\lfloor\log_3 k\rfloor, R(k)=4^m+k-3^m.$ Then,
\[
\ind_{\Z_p}(|I_{p,k}|)\ge R(k)-1.
\]
Consequently, $\ind_{\Z_p}(|I_{p,k}|)=\Omega(k^{\log_3 4})$.
\end{theorem}

For the remainder of this section, let $p$ be a prime and set $G:=\Z_p$. To prove the lower bound, we construct a continuous $G$-equivariant map
\[
 |E_{R(k)-1}G|\longrightarrow |I_{p,k}|.
\]
A map from $|I_{p,k}|$ to a smaller join would compose with this map to contradict Dold's theorem (\Cref{thm:dold}).  \Cref{sec:schemes} reduces the desired index lower bound to the construction of ``schemes'', which select elements from chains of sets.  \Cref{sec:amplification} introduces the product and increment operations on these schemes. \Cref{sec:base_scheme} gives an explicit $(4,3)$-scheme, applies the two operations to construct schemes for every $k\ge3$, and completes the proof of \Cref{thm:lower_bound}.

\subsection{Schemes and Reduction}
\label{sec:schemes}
Recall that the vertices of $E_{r-1}G$ are the pair \((i,g)\), where \(i\) is one of the \(r\) layers. We may therefore view a vertex as assigning a group label to a single layer. Accordingly, a face of $E_{r-1}G$ assigns group labels to a subset of the $r$ layers, and a vertex of the second subdivision $\sd^2(E_{r-1}G)$ corresponds to a chain of such partial labellings.

We construct an equivariant map by selecting certain layers and then copying their group labels. The following definition formalizes the selection of layers independently of the group labels.

Let $U=\{0,\ldots,r-1\}$ be a nonempty finite set with $r$ elements. Denote by $\mathcal{C}(U)$ the set of all nonempty strict chains of nonempty subsets of $U$:
\[
 \mathcal C=(C_1\subsetneq\cdots\subsetneq C_m),
 \qquad \varnothing\ne C_i\subseteq U.
\]
We regard each chain as the set of its members, thus a chain $\mathcal C$ is contained in another chain $\mathcal D$ if every member of $\mathcal C$ is also a member of $\mathcal D$. The largest member of $\mathcal C$ is denoted by $\max(\mathcal C)=C_m$.

\begin{definition}[$(r,q)$-Schemes]
\label{def:scheme}
Let $r,q\ge1$ be integers and set $U=\{0,\ldots,r-1\}$. An \emph{$(r,q)$-scheme} is an assignment $s:\mathcal{C}(U)\to U^q$ such that for every $\mathcal{C},\mathcal{D}\in\mathcal{C}(U)$,
\begin{itemize}
    \item (validity) for every $t\in[q]$, $s_t(\mathcal C)\in\max(\mathcal C)$;
       \item (consistency) if $\mathcal C\subseteq\mathcal D$, then $s_t(\mathcal C)=s_t(\mathcal D)$ for some $t\in[q];$
 \end{itemize}
where $s(\mathcal C)=(s_1(\mathcal C),\ldots,s_q(\mathcal C))$. The coordinates of a tuple are allowed to repeat.
\end{definition}

The following lemma reduces the desired index lower bound to the construction of schemes.

\begin{lemma}
\label{lem:scheme-map}
Let $r,q\geq1$ be integers. An $(r,q)$-scheme $s$ yields a simplicial $G$-equivariant map
\[
 \Phi:\sd^2(E_{r-1}G)\longrightarrow I_{p,q}.
\]
Consequently, it further yields a continuous $G$-equivariant map $|E_{r-1}G|\to|I_{p,q}|$.
\end{lemma}

\begin{proof}
Fix an $(r,q)$-scheme $s$ and let $U:=\{0,\ldots,r-1\}$. For a nonempty face $A\in E_{r-1}G$, denote its layer support by
\[
 \supp(A)=\{i\in U:\exists g\in G, (i,g)\in A\}.
\]
If $A\subseteq B$, then $\supp(A)\subseteq\supp(B)$.  Moreover, $A\subsetneq B$ implies $\supp(A)\subsetneq\supp(B)$ since each face of $E_{r-1}G$ contains at most one vertex in each layer.

\par\smallskip\noindent\textit{(i) Defining the map $\Phi$.} A vertex of $\sd^2(E_{r-1}G)$ is a nonempty strict chain $\mathcal A=(A_1\subsetneq\cdots\subsetneq A_m)$ of nonempty faces. It further gives
\[
 S(\mathcal A):=(\supp(A_1)\subsetneq\cdots\subsetneq\supp(A_m))
 \in\mathcal{C}(U).
\]
For every $t\in[q]$, we define the map $\Phi:\sd^2(E_{r-1}G)\to I_{p,q}$ such that
\begin{equation}
 \label{eq:decoding}
 \Phi(\mathcal A)_t:=g
 \quad\text{where }\bigl(s_t(S(\mathcal A)),g\bigr)\in \max(\mathcal{A}).
\end{equation}
This map is well defined:  By \Cref{def:scheme}, the layer $s_t(S(\mathcal A))$ belongs to $\supp(\max(\mathcal{A}))$, so such a $g\in G$ exists, and it is unique because  $\max(\mathcal{A})$ contains at most one vertex in each layer.

\par\smallskip\noindent\textit{(ii) $\Phi$ is a simplicial map.} Let $\mathcal A\subseteq\mathcal B$ be two vertices of $\sd^2(E_{r-1}G)$, where $\mathcal B=(B_1\subsetneq\cdots\subsetneq B_n)$.  It follows that $S(\mathcal A)\subseteq S(\mathcal B)$ as chains.  By \Cref{def:scheme}, there exists a $t\in[q]$ such that
\[
 s_t(S(\mathcal A))=s_t(S(\mathcal B)).
\]
Since the face $\max(\mathcal{A})$ is also a member of $\mathcal B$,  $\max(\mathcal{A})\subseteq \max(\mathcal{B})$. By \eqref{eq:decoding},
\[
 \bigl(s_t(S(\mathcal A)),\Phi(\mathcal A)_t\bigr)\in \max(\mathcal{A})\subseteq \max(\mathcal{B}).
\]
Combining with $s_t(S(\mathcal A))=s_t(S(\mathcal B))$ and the uniqueness of $\Phi(\mathcal B)_t$, we have
\begin{equation}\label{eq:two_chains_agree_at_one_coordinate}
    \Phi(\mathcal A)_t=\Phi(\mathcal B)_t.
\end{equation}

For every face of $\sd^2(E_{r-1}G)$,  any pair of its vertices $\mathcal{A},\mathcal{B}$ are comparable chains, so by \eqref{eq:two_chains_agree_at_one_coordinate} the images of the two vertices agree in a coordinate. The image of the face of $\sd^2(E_{r-1}G)$ is therefore a face of $I_{p,q}$ by \Cref{def:agreement-complex}. This proves that $\Phi$ is a  simplicial map $\sd^2(E_{r-1}G)\to I_{p,q}$.

\par\smallskip\noindent\textit{(iii) $\Phi$ is a $G$-equivariant map.} We define a $G$-action on $E_dG$ by
\[
g(i,h)=(i,g+h),
\]
for every $g\in G$ and vertex $(i,h)\in E_dG$. This action is simplicial on $E_dG$ and its affine extension is free on $|E_dG|$. We next define the action on $\sd^2(E_{r-1}G)$ by acting on every face in the chain for vertices of $\sd^2(E_{r-1}G)$. For every $h\in G$ and every vertex $\mathcal{A}$ of $\sd^2(E_{r-1}G)$, the action does not change layer supports, so $S(h\mathcal A)=S(\mathcal A)$.  For every  $t\in[q]$, by \eqref{eq:decoding},
\[
 \bigl(s_t(S(\mathcal A)),h+\Phi(\mathcal A)_t\bigr)=\bigl(s_t(S(h\mathcal A)),h+\Phi(\mathcal A)_t\bigr)\in h\max(\mathcal{A})
\]
This further gives
\[
 \Phi(h\mathcal A)_t=h+\Phi(\mathcal A)_t.
\]
Thus the map $\Phi$ is equivariant on vertices and hence naturally extended on its realization $|\Phi|$.

For any finite complex $K$, by \Cref{def:barycentric-subdivision}, \(|\sd K|\) is homeomorphic to \(|K|\). Thus $|\Phi|$ further yields a continuous $G$-equivariant map $|E_{r-1}G|\to|I_{p,q}|$.
\end{proof}

\begin{corollary}
\label{cor:scheme-index}
If an $(r,q)$-scheme exists, then for every prime $p$,
\[
 \ind_{\Z_p}(|I_{p,q}|)\ge r-1.
\]
\end{corollary}

\begin{proof}
If $r=1$, the bound follows from the nonnegativity of the index. Assume $r\ge2$ and is finite. By \Cref{lem:scheme-map}, an $(r,q)$-scheme gives a continuous equivariant map $f:|E_{r-1}G|\to |I_{p,q}|.$ The diagonal action on $|I_{p,q}|$ is free, as verified in \Cref{sec:metric-transfer}.  Suppose the index $\ind_{\Z_p}(|I_{p,q}|)=d<r-1$.  By \Cref{def:topological-index}, there is a continuous equivariant map $\psi:|I_{p,q}|\to|E_dG|$.  Since $d+1\le r-1$, inclusion of the first $d+2$ layers gives an equivariant map $\iota:|E_{d+1}G|\hookrightarrow|E_{r-1}G|$.  The composition of maps
\[
 |E_{d+1}G|\xrightarrow{\ \iota\ }|E_{r-1}G|
 \xrightarrow{\ f\ }|I_{p,q}|
 \xrightarrow{\ \psi\ }|E_dG|
\]
contradicts \Cref{thm:dold}.

The case of infinite index already satisfies the claimed bound.
\end{proof}

\subsection{Operations of Schemes}
\label{sec:amplification}

We introduce two operations for constructing new schemes from existing ones. The product operation multiplies the two parameters, while the increment operation increases both the number of source layers and the number of output coordinates by one. Using these two operations, we can extend a small base scheme, namely a $(4,3)$-scheme in \Cref{sec:base_scheme}, to schemes of any prescribed size.

\subsubsection{Product Operation}

\begin{lemma}
\label{lem:product}
An $(r,q)$-scheme and an $(r',q')$-scheme yield an $(rr',qq')$-scheme.
\end{lemma}
Let $U=\{0,\ldots,r-1\}$ and $V=\{0,\ldots,r'-1\}$. We construct an $(rr',qq')$-scheme by first selecting the component in $U$ using the $(r,q)$-scheme, and then applying the other scheme to the second components associated with that choice.

\begin{proof}
\textit{(i) Scheme construction.} We denote the given $(r,q)$-scheme and $(r',q')$-scheme by $s$ and $s'$, respectively. Fix a strict chain $\mathcal C=(C_1\subsetneq\cdots\subsetneq C_m)\in\mathcal{C}(U\times V)$. For a nonempty set $C\subseteq U\times V$, let the projected set
\[
 p(C):=\{u\in U:\exists v\in V, (u,v)\in C\}.
\]
We further define the collection of projected sets by
\[
P(\mathcal C)=\{p(C):C\in\mathcal C\}.
\]
This collection naturally forms a strict chain in $\mathcal{C}(U)$. In the absence of ambiguity, we use $P(\mathcal C)$ to denote this chain. For each $i\in[q]$, we use the $(r,q)$-scheme $s$ to choose the first coordinate by
\begin{equation}
 \label{eq:outer-choice}
 a_i(\mathcal{C}):=s_i(P(\mathcal C))\in\max(P(\mathcal C))=p(\max(\mathcal{C})).
\end{equation}

Now fix $u\in p(\max(\mathcal C))\subseteq U$. For a nonempty set $C\subseteq U\times V$, define the restriction of $C$ on $u$ as:
\[
q_u(C)=\{v\in V:(u,v)\in C\}.
\]
We further define the collection of restriction sets of $\mathcal{C}$ on $u$ by:
\[
Q_u(\mathcal C)
 :=\bigl\{q_u(C):C\in\mathcal C\bigr\}
   \setminus\{\varnothing\}.
\]
Empty sets are removed, and repetitions are identified. Thus this collection also naturally forms a strict chain in $V$. In the absence of ambiguity, we use $Q_u(\mathcal C)$ to denote this chain. For each $j\in[q']$, we use the $(r',q')$-scheme $s'$ to choose the second coordinate by
\begin{equation}
 \label{eq:inner-choice}
 b_j^{(u)}(\mathcal{C}):=s_j'(Q_u(\mathcal C))\in\max(Q_u(\mathcal C))=q_u(\max(\mathcal{C})).
\end{equation}

Therefore, for every $\mathcal{C}\in\mathcal{C}(U\times V)$ and every $(i,j)\in[q]\times[q']$, by \eqref{eq:outer-choice} and \eqref{eq:inner-choice}, we define
\begin{equation}
 \label{eq:product-rule}
 \tilde{s}_{i,j}(\mathcal C)=
 \left(a_i(\mathcal C),b_j^{(a_i(\mathcal C))}(\mathcal{C})\right).
\end{equation}

\par\smallskip\noindent\textit{(ii) Validity.} By \eqref{eq:inner-choice}, the second coordinate in \eqref{eq:product-rule} belongs to
\[
 \max (Q_{a_i(\mathcal C)}(\mathcal C))
 =\{v\in V:(a_i(\mathcal C),v)\in \max(\mathcal{C})\}.
\]
Hence the pair in \eqref{eq:product-rule} belongs to $\max(\mathcal{C})$ for every $(i,j)\in[q]\times[q']$.

\par\smallskip\noindent\textit{(iii) Consistency.} Suppose there are two chains $\mathcal C\subseteq\mathcal D$. For every $A\in P(\mathcal C)$, by definition there exists a set $C\in\mathcal{C}\subseteq\mathcal{D}$ such that $A=p(C)\in P(\mathcal{D}).$ Thus $P(\mathcal C)\subseteq P(\mathcal D).$ The consistency of $s$ in \Cref{def:scheme} implies a coordinate $i\in[q]$ such that
\begin{equation}
 \label{eq:common-outer}
 a_i(\mathcal C)=a_i(\mathcal D)=:a.
\end{equation}
The chains $Q_a(\mathcal C)$ and $Q_a(\mathcal D)$ are both nonempty by \eqref{eq:outer-choice}. Similarly, for every $B\in Q_a(\mathcal C)$, by definition there exists a set $C\in\mathcal{C}\subseteq\mathcal{D}$ such that $B=q_a(C)\in Q_a(\mathcal{D})$. Thus $Q_a(\mathcal C)\subseteq Q_a(\mathcal D).$ The consistency of $s'$ in \Cref{def:scheme} implies a coordinate $j\in[q']$ such that
\[
b_j^{(a)}(\mathcal C)=b_j^{(a)}(\mathcal D).
\]
Combined with \eqref{eq:common-outer} we have $\tilde{s}_{i,j}(\mathcal C)=\tilde{s}_{i,j}(\mathcal D)$ for some $(i,j)\in[q]\times[q']$.

Therefore, $\tilde{s}$ is a valid $(rr',qq')$-scheme.
\end{proof}

\subsubsection{Increment Operation}

\begin{lemma}
\label{lem:stabilization}
An $(r,q)$-scheme yields an $(r+1,q+1)$-scheme.
\end{lemma}
Let $U=\{0,\ldots,r-1\}$ and $z=r$. We construct an $(r+1,q+1)$-scheme on $U\cup\{z\}$ by applying the given scheme after deleting $z$ from each member of a chain. One additional coordinate handles chains containing the singleton $\{z\}$.

\begin{proof}
\textit{(i) Scheme construction.} We denote the given $(r,q)$-scheme by $s$. Fix a strict chain $\mathcal C=(C_1\subsetneq\cdots\subsetneq C_m) \in\mathcal{C}(U\cup\{z\})$. Define the collection of sets obtained by deleting $z$ as
\[
\mathcal C^-:=\{C\setminus\{z\}:C\in\mathcal C\}
 \setminus\{\varnothing\}.
\]
Empty sets are removed, and repetitions are identified. Since deletion preserves inclusions, a nonempty $\mathcal C^-$ naturally forms a strict chain in $\mathcal{C}(U)$, with $\max(\mathcal C^-)=\max(\mathcal C)\setminus\{z\}.$

If $\mathcal C^-=\varnothing$, then $\mathcal C=(\{z\})$. In this case, define $\tilde{s}_i(\mathcal C)=z$ for every $i\in[q+1]$. Otherwise, for every $i\in[q+1]$, define
\begin{equation}
 \label{eq:additional-coordinate}
 \tilde{s}_{i}(\mathcal C):=
 \begin{cases}
 s_i(\mathcal C^-),&i\ne q+1,\\
  z,&i= q+1,\{z\}\in\mathcal C,\\
  s_1(\mathcal C^-),&i= q+1,\{z\}\notin\mathcal C.
 \end{cases}
\end{equation}

\par\smallskip\noindent\textit{(ii) Validity.} If $\mathcal C^-=\varnothing$, then $\mathcal C=(\{z\})$ and all coordinates equal $z\in\max(\mathcal C)$. Otherwise, the validity of $s$ in \Cref{def:scheme} and \eqref{eq:additional-coordinate} give, for every $i\in[q]$,
\[
 \tilde{s}_i(\mathcal C)=s_i(\mathcal C^-)
 \in\max(\mathcal C^-)
 =\max(\mathcal C)\setminus\{z\}
 \subseteq\max(\mathcal C).
\]
By \eqref{eq:additional-coordinate}, the last coordinate either repeats $\tilde{s}_1(\mathcal C)$ or equals $z$ when $\{z\}\in\mathcal C$. Hence it also belongs to $\max(\mathcal C)$.

\par\smallskip\noindent\textit{(iii) Consistency.} Suppose there are two chains $\mathcal C\subseteq\mathcal D$. For every $A\in\mathcal C^-$, by definition there exists a set $C\in\mathcal C\subseteq\mathcal D$ such that $A=C\setminus\{z\}\ne\varnothing$. Thus $A\in\mathcal D^-$, and $\mathcal C^-\subseteq\mathcal D^-.$

If $\mathcal C^-\ne\varnothing$, then $\mathcal D^-\ne\varnothing$. The consistency of $s$ in \Cref{def:scheme} implies a coordinate $i\in[q]$ such that
\[
 \tilde{s}_i(\mathcal C)=s_i(\mathcal C^-)
 =s_i(\mathcal D^-)=\tilde{s}_i(\mathcal D).
\]

If $\mathcal C^-=\varnothing$, then $\mathcal C=(\{z\})$ and $\{z\}\in\mathcal D$. By construction, whether or not $\mathcal D^-$ is empty,
\[
 \tilde{s}_{q+1}(\mathcal C)=z=\tilde{s}_{q+1}(\mathcal D).
\]
Therefore, $\tilde{s}$ is a valid $(r+1,q+1)$-scheme.
\end{proof}

\subsection{Explicit Schemes and Proof of the Lower Bound}
\label{sec:base_scheme}

We first construct a $(4,3)$-scheme.  Repeated products and increments then give an $(R(k),k)$-scheme for every $k\ge3$.

\begin{lemma}
\label{lem:base}
There is a $(4,3)$-scheme.
\end{lemma}

\begin{proof}
\textit{(i) Scheme construction.} Let $U=\{0,1,2,3\}$. For every nonempty set $A\subseteq U$ and every $t\in[3]$, define
\begin{equation}
 \label{eq:baseline}
 \ell_t(A)=
 \begin{cases}
  t,&t\in A,\\
  \min\limits_{a \in A} a,&t\notin A.
 \end{cases}
\end{equation}

We call a chain $\mathcal C\in\mathcal{C}(U)$ \emph{special} if both $\{0\}$ and $U$ belong to $\mathcal C$. Define $s:\mathcal C(U)\to U^3$ by
\begin{equation}
 \label{eq:base-rule}
 s(\mathcal C)=
 \begin{cases}
  \ell(\max(\mathcal C)),&\mathcal C\text{ is not special},\\
  (1,0,3),&\mathcal C\text{ is special}, \{0,2\}\notin\mathcal C,\\
  (1,2,0),&\mathcal C\text{ is special}, \{0,2\}\in\mathcal C.
 \end{cases}
\end{equation}

\par\smallskip\noindent\textit{(ii) Validity.} By \eqref{eq:baseline}, $\ell_t(A)\in A$ for every nonempty set $A\subseteq U$ and every $t\in[3]$. Thus, if $\mathcal C$ is not special, for every $t\in[3]$,
\[
 s_t(\mathcal C)=\ell_t(\max(\mathcal C))\in\max(\mathcal C).
\]
If $\mathcal C$ is special, then $\max(\mathcal C)=U$, and each coordinate of $s(\mathcal{C})$ belongs to $U$.

\par\smallskip\noindent\textit{(iii) Consistency.} Suppose there are two chains $\mathcal C\subseteq\mathcal D$. Let $A=\max(\mathcal C)$ and $B=\max(\mathcal D)$. We consider the following three cases.

\par\smallskip\noindent\textit{Case 1. Neither chain is special.} If $A\cap[3]\ne\varnothing$, choose $t\in A\cap[3]$. Since $A\subseteq B$, \eqref{eq:baseline} and \eqref{eq:base-rule} give
\[
 s_t(\mathcal C)=\ell_t(A)=t=\ell_t(B)=s_t(\mathcal D).
\]
Otherwise $A=\{0\}$. Since $A\in\mathcal{C}\subseteq\mathcal D$ and $\mathcal D$ is not special, $B\ne U$. Thus there exists $t\in[3]\setminus B$. Together with $\min A=\min B=0$, this gives
\[
 s_t(\mathcal C)=\ell_t(A)=0=\ell_t(B)=s_t(\mathcal D).
\]

\par\smallskip\noindent\textit{Case 2. $\mathcal C$ is not special and $\mathcal D$ is special.} Since $\{0\}\in\mathcal D$, every member of $\mathcal D$ contains $0$. In particular, since $A\in\mathcal{C}\subseteq\mathcal{D}$, $0\in A$ and $\min A=0$.

If $\{0,2\}\notin\mathcal D$, then \eqref{eq:base-rule} gives $s(\mathcal D)=(1,0,3)$. If $2\notin A$, then $s_2(\mathcal C)=0=s_2(\mathcal D)$. Otherwise $\{0,2\}\subseteq A$. Since $A\in\mathcal D$ and $\{0,2\}\notin\mathcal D$, we have $A\ne\{0,2\}$. Hence there exists $t\in A\cap\{1,3\}$, and $s_t(\mathcal C)=t=s_t(\mathcal D)$, since $U\in\mathcal{D}$.

If $\{0,2\}\in\mathcal D$, then $s(\mathcal D)=(1,2,0)$. If $3\notin A$, then $s_3(\mathcal C)=0=s_3(\mathcal D)$. Otherwise $\{0,3\}\subseteq A$. Since $A$ and $\{0,2\}$ are members of the same chain, they are comparable, so $2\in A$. Hence $s_2(\mathcal C)=2=s_2(\mathcal D)$, since $U\in\mathcal{D}$.

\par\smallskip\noindent\textit{Case 3. $\mathcal C$ is special.} Since $\mathcal C\subseteq\mathcal D$, both $\{0\}$ and $U$ belong to $\mathcal D$, so $\mathcal D$ is also special. By \eqref{eq:base-rule},
\[
 s_1(\mathcal C)=1=s_1(\mathcal D).
\]
The three cases prove the consistency of $s$ in \Cref{def:scheme}. Therefore, $s$ is a valid $(4,3)$-scheme.
\end{proof}
\begin{corollary}
\label{cor:all-k}
For every integer $k\ge3$, let $m=\lfloor\log_3 k\rfloor$. Then there is an $(R(k),k)$-scheme, where
\[
 R(k)=4^m+k-3^m.
\]
\end{corollary}

\begin{proof}

Now fix $k\ge3$ and let $m=\lfloor\log_3 k\rfloor\ge1$. The definition of $m$ gives $3^m\le k<3^{m+1}$, so by \Cref{lem:product} and \Cref{lem:base}, we can construct a $(4^m,3^m)$-scheme by induction.

Let $h=k-3^m$. By applying \Cref{lem:stabilization}, we obtain a $(4^m+h,3^m+h)=(R(k),k)$-scheme, where $R(k)=4^m+k-3^m.$
\end{proof}

We are now ready to prove the lower bound of the agreement index.

\begin{proof}[Proof of \Cref{thm:lower_bound}]
Fix a prime $p$ and an integer $k\ge3$.  Let $m=\lfloor\log_3 k\rfloor$ and $R(k)=4^m+k-3^m$. By \Cref{cor:all-k} there is an $(R(k),k)$-scheme. \Cref{cor:scheme-index} therefore gives
\[
 \ind_{\Z_p}(|I_{p,k}|)\ge R(k)-1.
\]

Consequently, by $m=\lfloor\log_3 k\rfloor$ and $3^m\le k<3^{m+1}$,
\[
R(k)-1=4^m+k-3^m-1
 \ge4^m-1
 \ge\frac34\,4^m
 >\frac{3}{16}\,k^{\log_3 4}.
\]
Combining the above bounds proves the claimed $\Omega(k^{\log_3 4})$ lower bound.
\end{proof}

\section{Reductions between Agreement Indices}
\label{sec:reduction}

In this section, we establish a limitation of the topological method introduced in \Cref{sec:metric-transfer}. We show that neither replacing the underlying topological space nor modifying the labelling method can yield a stronger lower bound for tree covers (\Cref{thm:reduction}). As stated in \Cref{prop: superlinear}, this lemma, together with \Cref{thm:lower_bound}, shows that algebraic/topological approaches similar to the one in \Cref{sec:metric-transfer} cannot establish a tree-cover lower bound stronger than $n^{O(1/k^{\log_3 4})}$ for every $k>0$.

Throughout this section, let $H$ be a finite nontrivial group. We first extend the topological index (\Cref{def:topological-index}) by allowing any finite free complexes as targets.

\begin{definition}[Extended Topological Index]
\label{def:extended-index}
Let $D\ge1$ be an integer, and let $\mathcal K=(K_d)_{d\ge D}$ be a family of finite ASCs. Suppose that, for every $d\ge D$, $\dim K_d\le d$, and $K_d$ has a simplicial $H$-action that is free on $|K_d|$. For a topological space $X$ with a free $H$-action, its \emph{extended topological index} with respect to $\mathcal K$ is
\begin{equation*}
    \ind_{\mathcal K}^{(H)}(X)
    :=\inf\{d\ge D:\text{there is a continuous $H$-equivariant map } X\to|K_d|\},
\end{equation*}
with value $+\infty$ if the set is empty.
\end{definition}

We also write $\ind_G^{(G)}(X):=\ind_G(X)$ for the usual topological index from \Cref{def:topological-index}.

Let $A_1,\ldots,A_k$ be finite alphabets, and let $X:=|I(A_1,\ldots,A_k)|$ denote the corresponding agreement complex. By an argument similar to that in \Cref{sec:metric-transfer}, one obtains a lower bound of $\Omega_k(n^{1/(\ind_{\mathcal K}^{(H)}(X)+1)})$ for tree covers. The choice of $\mathcal K$ determines the underlying topological space used in the hard instance, while the alphabets $A_1,\ldots,A_k$ determine the labelling method for the $k$ trees. We show that this extended index is bounded below by $\ind_{\Z_p}^{(\Z_p)}(|I_{p,k}|)$ for every prime $p\mid |H|$.

\begin{theorem}
\label{thm:reduction}
Let $D,k\ge1$ be integers, and let $\mathcal K=(K_d)_{d\ge D}$ be a family of finite ASCs. Suppose that, for every $d\ge D$, $\dim K_d\le d$, and $K_d$ has a simplicial $H$-action that is free on $|K_d|$. Let $A_1,\ldots,A_k$ be finite alphabets. Suppose that $H$ acts simplicially on $I(A_1,\ldots,A_k)$ and freely on $X:=|I(A_1,\ldots,A_k)|$. For every prime $p\mid |H|$,
\[
\ind_{\mathcal K}^{(H)}(X)\ge\ind_H^{(H)}(X)
    \ge\ind_{\Z_p}^{(\Z_p)}(|I_{p,k}|).
\]
\end{theorem}

We prove the two inequalities separately. \Cref{sec:reduction-space} proves the first by mapping each $|K_d|$ to $|E_dH|$. For the second, \Cref{sec:reduction-label} reduces the actions to monomial actions, reduces the alphabets, and finally completes the proof of \Cref{thm:reduction}.

\subsection{Reduction on Topological Spaces}
\label{sec:reduction-space}

\begin{proposition}
\label{prop:index-comparison}
Let $D\ge1$ be an integer, and let $\mathcal K=(K_d)_{d\ge D}$ be a family of finite ASCs. Suppose that, for every $d\ge D$, $\dim K_d\le d$, and $K_d$ has a simplicial $H$-action that is free on $|K_d|$. Let $X$ be a topological space with a free $H$-action. Then
\[
\ind_{\mathcal K}^{(H)}(X)\ge\ind_H^{(H)}(X).
\]
\end{proposition}

\begin{proof}
If $\ind_{\mathcal K}^{(H)}(X)=+\infty$, the inequality holds. Otherwise, let $d=\ind_{\mathcal K}^{(H)}(X)\ge D$. By \Cref{def:extended-index}, there is a continuous $H$-equivariant map $f:X\to|K_d|$. By \Cref{lem:free-complex-to-join}, there is a continuous $H$-equivariant map $\alpha_d:|K_d|\to|E_dH|$. The composition
\[
    X\xrightarrow{\ f\ }|K_d|\xrightarrow{\ \alpha_d\ }|E_dH|
\]
is therefore continuous and $H$-equivariant. By \Cref{def:topological-index}, $\ind_H^{(H)}(X)\le d$.
\end{proof}

\subsection{Reduction on Labelling Methods}
\label{sec:reduction-label}
The proof uses the monomial actions that permute the coordinates and the letters within each alphabet.

\begin{definition}[Monomial Action]
\label{def:monomial-action}
Let $k\ge1$ be an integer, and let $A_1,\ldots,A_k$ be finite alphabets. A \emph{monomial $H$-action} consists of permutations $\pi_g$ of $[k]$ and bijections $\beta_{g,i}:A_i\to A_{\pi_g(i)}$ such that for every $g,h\in H,\ i\in[k]$,
\begin{equation*}
\begin{aligned}
    \pi_e&=\id_{[k]}, & \pi_{gh}&=\pi_g\circ\pi_h,\\
    \beta_{e,i}&=\id_{A_i}, &
    \beta_{gh,i}&=\beta_{g,\pi_h(i)}\circ\beta_{h,i}.
\end{aligned}
\end{equation*}
The action on words $x:=(x_1,\ldots,x_k)$ is given by
\[
(gx)_{\pi_g(i)}=\beta_{g,i}(x_i).
\]
\end{definition}

One can verify that $(gh)x=g(hx)$, so this action is well-defined. Moreover, if $x_i=y_i$, then $(gx)_{\pi_g(i)}=(gy)_{\pi_g(i)}$. Thus this action preserves agreement between two words and defines a simplicial action on $I(A_1,\ldots,A_k)$.

For every subgroup $P\le H$ of prime order $p$, we define the $P$-action by restricting the $H$-action in \Cref{thm:reduction} to $P$. This action is still simplicial on $I(A_1,\ldots,A_k)$ and free on its realization.

We show that this $P$-action can be conjugated to a monomial action. The proof is given in \Cref{sec:monomial-reduction}.

\begin{restatable}{lemma}
{PrimeOrderConjugacy}
\label{lem:prime-conjugacy}
Let $P$ be a group of prime order acting simplicially on $I:=I(A_1,\ldots,A_k)$ and freely on $|I|$ by $\rho:P\times |I|\to |I|$. There is a simplicial automorphism $F$ of $I$ such that the conjugated action $\zeta:P\times V(I)\to V(I)$, defined on the vertices by
\[
    \zeta(g,x):=F(\rho(g,F^{-1}(x)))
\]
for every $g\in P$ and $x\in V(I)$, is monomial.
\end{restatable}

We next prove that a free monomial $P$-action fixes each coordinate position and acts freely on each alphabet.

\begin{lemma}
\label{lem:prime-rigidity}
Let $P$ be a group of prime order acting simplicially on $I:=I(A_1,\ldots,A_k)$. Suppose this action is monomial and free on $|I|$. Then it fixes each coordinate position, i.e., $\pi_g=\id_{[k]}$ for every $g\in P$, and acts freely on each alphabet $A_i$.
\end{lemma}

\begin{proof}
Write $P=\langle g\rangle$, where $g$ has prime order $p$. For $i\in[k]$ and $a\in A_i$, let $W_{i,a}:=\{x\in A_1\times \ldots\times A_k:x_i=a\}$ denote the set of words whose $i$-th coordinate is $a$. Let $D_{i,a}$ denote the simplicial complex $\Delta_{W_{i,a}}$.

\par\smallskip\noindent\textit{(i) $P$ fixes coordinate positions.} Suppose, for contradiction, that $\pi_g$ has a nontrivial orbit with length $\ell$. Since $g^p=e$, we have $(\pi_{g})^p=\pi_e=\id_{[k]}$ and hence $\ell\mid p$. As $p$ is prime and the orbit is nontrivial, it follows that $\ell=p$. Choose a position $i$ in this orbit and a letter $a\in A_i$. The intersection complex
\[
    L:=D_{i,a}\cap gD_{i,a}\cap\cdots\cap g^{p-1}D_{i,a}=\bigcap_{j=0}^{p-1}D_{\pi_g^j(i),\beta_{g^j,i}(a)}=\Delta_{W_L},
\]
where $W_L:=\{x\in A_1\times \ldots\times A_k:x_{\pi_g^j(i)}=\beta_{g^j,i}(a)\text{ for } \forall j=0,1,\ldots,p-1\}$. Since the positions
$\pi_g^j(i)$ are pairwise distinct, these constraints are compatible. Therefore the complex is nonempty. We denote the barycenter as $c$ and the vertex set of $L$ as $V(L)$. Then $c=\frac{1}{|V(L)|}\sum_{x\in V(L)} x$. By $g^{p}=e$ we have $gL=L$. Thus
\[
gc=\frac{1}{|V(L)|}\sum_{x\in V(L)}g x=\frac{1}{|V(L)|}\sum_{y\in gV(L)} y=c.
\]
Since $|L|\subseteq |I(A_1,\ldots,A_k)|$, this contradicts with the freeness of $P$-action on $|I(A_1,\ldots,A_k)|$.
Hence $\pi_g=\id_{[k]}$.

\par\smallskip\noindent\textit{(ii) $P$ acts freely on $A_i$.} The maps $\beta_{g,i}$ now define a $P$-action on each $A_i$. Suppose, for contradiction, that this action is not free. Then there exists $a\in A_i$ such that $\beta_{g,i}(a)=a$. It follows that $gD_{i,a}=D_{i,a}$. Denote its barycenter as $c$. By the same argument as above, we have $gc=c$, contradicting the freeness of $P$-action on $|I(A_1,\ldots,A_k)|$. Hence $P$ acts freely on each $A_i$.
\end{proof}

We use the notation $I_{P,k}$ from \Cref{def:agreement-complex}. Let $P$ act diagonally on $I_{P,k}$ by
\[
    g(x_1,\ldots,x_k)=(gx_1,\ldots,gx_k),
\]
for every $g\in P$ and $(x_1,\ldots,x_k)\in P^k$. We now prove the existence of the simplicial $P$-equivariant map from $I_{P,k}$ to $I(A_1,\ldots,A_k)$.

\begin{lemma}
\label{lem:alphabet}
Let $P$ be a finite group, and let $A_1,\ldots,A_k$ be nonempty finite sets, each equipped with a free $P$-action. Then there is a simplicial $P$-equivariant map
\[
s: I_{P,k}\longrightarrow I(A_1,\ldots,A_k).
\]
\end{lemma}

We construct $s$ by selecting one orbit in each alphabet.

\begin{proof}
For each $i\in[k]$, choose a representative $a_{i,t}$ of every orbit of $P$ in $A_i$ where $t=0,\ldots,T-1$ and $T$ is the number of orbits. Define $s_i(h):=h a_{i,0}$ for every $h\in P$. For every vertex $x=(x_1,\ldots,x_k)\in P^k$ of $I_{P,k}$, apply these maps in each coordinate:
\[
s(x_1,\ldots,x_k):=(s_1(x_1),\ldots,s_k(x_k)).
\]

If two words $x,y$ agree in a coordinate, $s(x)$ and $s(y)$ also agree in that coordinate. Thus the image of each face of $I_{P,k}$ is still a face of $I(A_1,\ldots,A_k)$, hence $s$ is simplicial. For every $g,h\in P$, $s_i(gh)=gha_{i,0}=g s_i(h)$. Hence $s$ is $P$-equivariant.
\end{proof}

To compare the indices, we also map an $H$-join to a $P$-join after restricting the action to $P$.

\begin{lemma}
\label{lem:subgroup-join}
For every subgroup $P\le H$ and every integer $d\ge0$, there is a simplicial $P$-equivariant map $\gamma_d:E_dH\to E_dP$, where $P$ acts on $E_dH$ by restricting the $H$-action.
\end{lemma}

\begin{proof}
Choose a representative $t$ of each orbit $Pt$ of the $P$-action on $H$. Every element of $H$ has a unique expression $at$ with $a\in P$ and one of the chosen representatives $t$. Define $q:H\to P$ by $q(at):=a$. On the vertices of $E_dH$, define $\gamma_d(i,h):=(i,q(h)).$

This map preserves layers, so the image of each face of $E_dH$ is still a face of $E_dP$, hence $\gamma_d$ is simplicial. For every $g\in P$, $q(g(at))=q((ga)t)=ga=gq(at).$ Hence $\gamma_d$ is $P$-equivariant.
\end{proof}

The realization $|\gamma_d|:|E_dH|\to|E_dP|$ is therefore a continuous $P$-equivariant map.

\begin{proof}[Proof of \Cref{thm:reduction}]
Let $I:=I(A_1,\ldots,A_k)$ and $X:=|I|$. \Cref{prop:index-comparison} proves the first inequality.

If $\ind_H^{(H)}(X)=+\infty$, the second inequality also holds. Otherwise, let $d=\ind_H^{(H)}(X)$. By \Cref{def:topological-index}, there is a continuous $H$-equivariant map $f:X\to|E_dH|$.

For every prime $p\mid |H|$, there is a subgroup $P\le H$ of order $p$. We define the $P$-action $\rho$ by restricting the $H$-action to $P$. Thus $\rho$ is simplicial on $I$ and free on $X$. By \Cref{lem:prime-conjugacy}, there is a simplicial automorphism $F$ such that the conjugated action, such that, for every $g\in P,\ x\in V(I)$,
\[
    \zeta(g,x):=F(\rho(g,F^{-1}(x))),
\]
is monomial. By \Cref{lem:prime-rigidity}, the conjugated action $\zeta$ is coordinatewise and free on each $A_i$. Thus \Cref{lem:alphabet} gives a simplicial $P$-equivariant map $s:I_{P,k}\to I$ where $P$ acts diagonally on $I_{P,k}$ and acts freely on $I$ by $\zeta$. Consequently, $t:=F^{-1}\circ s$ is a simplicial $P$-equivariant map for the original action $\rho$, since for every $g\in P$ and $x\in P^k$,
\[
\begin{aligned}
    t(gx)=F^{-1}(s(gx))=F^{-1}(\zeta(g,s(x)))
         =\rho(g,F^{-1}(s(x)))=\rho(g,t(x)).
\end{aligned}
\]
Restricting $f$ to a $P$-action and applying \Cref{lem:subgroup-join} gives the composition
\[
|I_{P,k}|\xrightarrow{\ |t|\ }X
    \xrightarrow{\ f\ }|E_dH|
    \xrightarrow{\ |\gamma_d|\ }|E_dP|.
\]
Every map in this composition is continuous and $P$-equivariant and all the actions are free, so $\ind_P^{(P)}(|I_{P,k}|)\le d=\ind_H^{(H)}(X)$.

Since $\Z_p$ is isomorphic to $P$, $\ind_{\Z_p}^{(\Z_p)}(|I_{p,k}|)=\ind_P^{(P)}(|I_{P,k}|)$. This proves the second inequality for every prime $p\mid |H|$.
\end{proof}

\section*{Statement on the Use of AI}

We used OpenAI's ChatGPT (GPT-5.6, GPT-6 Astra) and Codex during the research and preparation of this paper. In developing the warm-up approach in \Cref{sec: cyclic}, ChatGPT pointed out that Borsuk-Ulam-type theorems extend beyond antipodal symmetry to more general finite group actions (\Cref{thm:MV_warmup}). The concrete hard-instance construction and the probabilistic argument for extracting a witness pair were developed by the authors. In subsequent AI-assisted explorations, the use of the agreement index was developed, reducing the tree-cover lower-bound problem to the computation of a topological invariant, as presented in \Cref{sec:metric-transfer}. The AI tools also made substantial contributions to the development of the quadratic index upper bound in \Cref{sec:index-upper-bound}, the superlinear index lower bound in \Cref{sec:lower_bound}, and the reductions between different indices in \Cref{sec:reduction}.

AI tools were additionally used to assist with figure generation and the preparation of initial drafts of the manuscript. The authors carefully verified the mathematical arguments and substantially rewrote several sections, with AI assistance in phrasing and editing. The authors take full responsibility for the correctness of all results and the final presentation of this paper.

\section*{Acknowledgements}

Shengtang Huang would like to thank Boyuan Hong and Xiaowei Ye for many helpful discussions, as well as for carefully reading an earlier version of this work and providing valuable suggestions.

\bibliographystyle{alpha}
\bibliography{refer}

\appendix
\section{Equivalence between \Cref{thm:dold} and \Cref{lem:selfmap}}
\label{sec:equiv_topo_lemmas}
\begin{proof}
Suppose \Cref{lem:selfmap} does not hold. Then there is a continuous map \(F:C(E_dG,v)\to |E_dG|\), where $v$ is the apex of $C(E_dG,v)$, such that its restriction to the base $|E_dG|$, denoted by $f:=F|_{|E_dG|}$, is $G$-equivariant. Assume $F(v)=y_0\in|E_dG|$. $|E_{d+1}G|$ can be seen as the union of cones on \(E_dG\) for every \(g\in G\), since
\[
|E_{d+1}G|=|\mathcal{D}(G)*E_{d}G|=\bigcup_{g\in G}|\Delta_{\{g\}}*E_dG|=\bigcup_{g\in G}C(E_dG,g).
\]
Thus for every point $z$ in $|E_{d+1}G|$, there is $x\in|E_{d}G|$ and $g\in G$, $t\in[0,1]$ such that $z=((1-t)x,tg)$.  On the cone with apex \(g\), define $\Phi_g(x,t):=gF((1-t)g^{-1}x+tv)$. We further define the map $\Phi:|E_{d+1}G|\to |E_dG|$ by $\Phi(z):=\Phi_g(x,t)$ for every $z=((1-t)x,tg)\in |E_{d+1}G|$. Since $f$ is $G$-equivariant, for every $g\in G$,
\[
\Phi_g(x,0)=gF(g^{-1}x)=gf(g^{-1}x)=f(x).
\]
On the other hand, for every $x\in |E_dG|$,
\[
\Phi_g(x,1)=gF(v)=gy_0.
\]
Therefore, $\Phi$ is well-defined and continuous.

\(\Phi\) is $G$-equivariant since for every $h\in G$ and every $z=((1-t)x,tg)\in|E_{d+1}G|$,
\[
h\Phi(z)=h\Phi_g(x,t)=hgF((1-t)g^{-1}x+tv)=hgF((1-t)(hg)^{-1}(hx)+tv)=\Phi_{hg}(hx,t)=\Phi(hz).
\]
Hence $\Phi$ is a continuous \(G\)-equivariant map \(|E_{d+1}G|\to |E_dG|\), contradicting \Cref{thm:dold}.

Suppose \Cref{thm:dold} does not hold, i.e., that there is a continuous \(G\)-equivariant map \(H:|E_{d+1}G|\to |E_dG|.\) Fix some $g\in G$. The cone $C(E_dG,g)$ is a subcomplex of $|E_{d+1}G|$. Thus a restriction of $H$ to the cone $C(E_dG,g)$ is a continuous map $C(E_dG,g)\to |E_dG|$. Since $g|E_d{G}|=|E_d{G}|$, the restriction of this map to its base $|E_dG|$ is $G$-equivariant, contradicting \Cref{lem:selfmap}.
\end{proof}

\section{Proof of \Cref{lem:mesh}}
\label{sec:triangulation}

We restate \Cref{lem:mesh} here. Recall that
\[
K:=E_dG=(\mathcal{D}(G))^{*(D)}
\qquad
\text{and}
\qquad
C_r:=rC(K)=\{tx:0\le t\le r,\ x\in|K|\}\subseteq\R^{V(K)}.
\]

\EquivariantTriangulation*

\begin{proof}
Recall that $C(K)$ carries a canonical triangulation induced by the simplicial structure of $K$.  For every face $F=\{v_1,\ldots,v_s\}\in K$, the corresponding simplices are
\[
    \conv\{v_1,\ldots,v_s\}
    \qquad\text{and}\qquad
    \conv\{0,v_1,\ldots,v_s\},
\]
together with their faces.  After scaling by $r$, these simplices give a fixed triangulation of $C_r$.  We construct $\Gamma_r$ by refining this triangulation.

Fix an $s$-dimensional simplex $\sigma=\conv\{v_0,\ldots,v_s\}$ of the canonical triangulation of $C(K)$.  Every point $z\in r\sigma$ has a unique representation
\[
    z=\sum_{i=0}^s \lambda_i v_i,
    \qquad
    \lambda_i\ge0,
    \qquad
    \sum_{i=0}^s\lambda_i=r.
\]
For every $i\in\{0,\ldots,s\}$ and every integer $j\in\{0,\ldots,r\}$, cut $r\sigma$ by the hyperplane
\[
    \lambda_i=j.
\]
The resulting polytopes, together with all of their faces, form a polytopal subdivision $\mathcal P_r(\sigma)$ of $r\sigma$.  These local subdivisions agree on common faces, since restricting the barycentric grid of a simplex to one of its faces gives exactly the grid obtained by applying the same construction directly to that face.  Hence they glue to a finite polytopal subdivision $\mathcal P_r$ of $C_r$.

The construction is $G$-invariant.  The $G$-action sends every simplex of the canonical cone triangulation to another simplex and permutes its vertices, and therefore permutes the corresponding barycentric coordinates and the hyperplanes $\lambda_i=j$.  Thus $G$ permutes the polytopes of $\mathcal P_r$.

\begin{figure}[h]
    \centering
    \begingroup
    \definecolor{gridpeach}{RGB}{255,228,198}
    \definecolor{gridgreen}{RGB}{236,250,219}
    \definecolor{gridblue}{RGB}{210,235,248}
    \definecolor{gridpink}{RGB}{246,218,232}
    \definecolor{gridyellow}{RGB}{255,251,203}
    \definecolor{gridrose}{RGB}{250,220,221}
    \definecolor{gridlavender}{RGB}{232,221,246}
    \definecolor{gridcyan}{RGB}{217,244,245}
    \definecolor{gridmint}{RGB}{226,246,232}
    \tikzset{
        grid original edge/.style={draw=black,line width=0.75pt},
        grid new edge/.style={draw=blue!80!black,line width=0.58pt},
        grid new vertex/.style={fill=red!80,draw=black,line width=0.25pt}
    }
    \subfigure[The $r=3$ subdivision of a triangle. The cuts $\lambda_i=j$ divide the simplex into nine smaller triangles. Original vertices and edges are shown in black, while the new grid vertices and subdivision edges are shown in red and blue, respectively.]
    {
        \resizebox{0.43\textwidth}{!}{\begin{tikzpicture}[line cap=round,line join=round]
            \path[use as bounding box] (0,0) rectangle (4.8,4.0);
            \foreach \i in {0,...,3}{
                \pgfmathtruncatemacro{\lastj}{3-\i}
                \foreach \j in {0,...,\lastj}{
                    \coordinate (g-\i-\j) at
                        ({0.35+4.10*(\i+0.5*\j)/3},{0.20+4.10*sqrt(3)*\j/6});
                }
            }
            \fill[gridlavender] (g-0-0) -- (g-1-0) -- (g-0-1) -- cycle;
            \fill[gridpink] (g-1-0) -- (g-0-1) -- (g-1-1) -- cycle;
            \fill[gridpeach] (g-1-0) -- (g-2-0) -- (g-1-1) -- cycle;
            \fill[gridcyan] (g-2-0) -- (g-1-1) -- (g-2-1) -- cycle;
            \fill[gridmint] (g-2-0) -- (g-3-0) -- (g-2-1) -- cycle;
            \fill[gridgreen] (g-0-1) -- (g-1-1) -- (g-0-2) -- cycle;
            \fill[gridblue] (g-1-1) -- (g-0-2) -- (g-1-2) -- cycle;
            \fill[gridyellow] (g-1-1) -- (g-2-1) -- (g-1-2) -- cycle;
            \fill[gridrose] (g-0-2) -- (g-1-2) -- (g-0-3) -- cycle;
            \foreach \i in {1,2}{
                \pgfmathtruncatemacro{\rest}{3-\i}
                \draw[grid new edge] (g-0-\i) -- (g-\rest-\i);
                \draw[grid new edge] (g-\i-0) -- (g-\i-\rest);
                \draw[grid new edge] (g-\i-0) -- (g-0-\i);
            }
            \draw[grid original edge] (g-0-0) -- (g-3-0) -- (g-0-3) -- cycle;
            \foreach \i/\j in {1/0,2/0,0/1,1/1,2/1,0/2,1/2}{
                \filldraw[grid new vertex] (g-\i-\j) circle (1.6pt);
            }
            \foreach \i/\j in {0/0,3/0,0/3}{
                \fill[black] (g-\i-\j) circle (2.1pt);
            }
        \end{tikzpicture}}
    }
    \hfill
    \subfigure[The $r=2$ subdivision of a tetrahedron. The cuts $\lambda_i=1$ cut off a smaller tetrahedron at each of the four original vertices, leaving a central octahedron whose six vertices are the edge midpoints. The four corner tetrahedra are shown in different colors, while the central octahedron is shown in yellow.]
    {
        \resizebox{0.43\textwidth}{!}{\begin{tikzpicture}[line cap=round,line join=round]
            \path[use as bounding box] (0,0) rectangle (4.8,4.0);
            \coordinate (A) at (2.10,3.70);
            \coordinate (B) at (0.35,0.89);
            \coordinate (C) at (3.15,0.15);
            \coordinate (D) at (4.45,1.27);
            \foreach \first/\second/\midpoint in {
                A/B/AB,A/C/AC,A/D/AD,B/C/BC,B/D/BD,C/D/CD}{
                \path (\first) -- coordinate[midway] (\midpoint) (\second);
            }
            \fill[gridpeach] (A) -- (AB) -- (AC) -- cycle;
            \fill[gridpeach] (A) -- (AC) -- (AD) -- cycle;
            \fill[gridgreen] (B) -- (BC) -- (AB) -- cycle;
            \fill[gridblue] (C) -- (AC) -- (BC) -- cycle;
            \fill[gridblue] (C) -- (CD) -- (AC) -- cycle;
            \fill[gridpink] (D) -- (AD) -- (CD) -- cycle;
            \fill[gridyellow] (AB) -- (BC) -- (AC) -- cycle;
            \fill[gridyellow] (AC) -- (CD) -- (AD) -- cycle;
            \draw[grid original edge,dashed,draw=black!60] (B) -- (D);
            \foreach \first/\second in {AB/AD,AB/BD,AD/BD,BC/BD,BD/CD,BC/CD}{
                \draw[grid new edge,dashed] (\first) -- (\second);
            }
            \foreach \first/\second in {AB/AC,AB/BC,AC/BC,AC/AD,AC/CD,AD/CD}{
                \draw[grid new edge] (\first) -- (\second);
            }
            \draw[grid original edge] (A) -- (B) -- (C) -- (D) -- cycle;
            \draw[grid original edge] (A) -- (C);
            \foreach \point in {AB,AC,AD,BC,BD,CD}{
                \filldraw[grid new vertex] (\point) circle (1.6pt);
            }
            \foreach \point in {A,B,C,D}{
                \fill[black] (\point) circle (2.1pt);
            }
        \end{tikzpicture}}
    }
    \endgroup
    \caption{Examples of the polytopal subdivisions induced by the cuts $\lambda_i=j$.}
    \label{fig:barycentric-grid-subdivision}
\end{figure}

We now take the barycentric subdivision of the polytopal complex $\mathcal P_r$.  For every nonempty polytope $P\in\mathcal P_r$, let
\[
    c(P):=\frac{1}{|V(P)|}\sum_{v\in V(P)}v
\]
be the barycenter of its vertices.  For every strict chain
\[
    P_0\subsetneq P_1\subsetneq\cdots\subsetneq P_t
\]
of nonempty polytopes in $\mathcal P_r$, take the simplex
\[
    \conv\{c(P_0),\ldots,c(P_t)\}.
\]
These simplices form a triangulation $\Gamma_r$ of $C_r$.  Since $\mathcal P_r$ refines the canonical cone triangulation, every original cone face and the base $r|K|$ is a subcomplex of $\Gamma_r$.  Moreover,
\[
    c(gP)=g\,c(P),
\]
so $G$ sends every simplex of $\Gamma_r$ to another simplex.  This proves the first two properties.

\begin{figure}[h]
    \centering
    \begingroup
    \tikzset{
        bary original edge/.style={draw=black,line width=0.75pt},
        bary new edge/.style={draw=blue!75,line width=0.58pt,dash pattern=on 2.2pt off 2.2pt},
        bary new vertex/.style={fill=red!80,draw=black,line width=0.25pt}
    }
    \subfigure[Barycentric subdivision of a triangle. Original vertices and edges are shown in black, while newly introduced barycenters and subdivision edges are shown in red and blue, respectively.]
    {
        \resizebox{0.43\textwidth}{!}{\begin{tikzpicture}[line cap=round,line join=round]
            \path[use as bounding box] (0,0) rectangle (4.8,4.0);
            \coordinate (A) at (0.35,0.20);
            \coordinate (B) at (4.45,0.20);
            \coordinate (C) at (2.40,{0.20+4.10*sqrt(3)/2});
            \foreach \first/\second/\midpoint in {A/B/AB,B/C/BC,C/A/CA}{
                \path (\first) -- coordinate[midway] (\midpoint) (\second);
            }
            \coordinate (G) at (barycentric cs:A=1,B=1,C=1);
            \foreach \point in {A,B,C,AB,BC,CA}{
                \draw[bary new edge] (G) -- (\point);
            }
            \draw[bary original edge] (A) -- (B) -- (C) -- cycle;
            \foreach \point in {AB,BC,CA,G}{
                \filldraw[bary new vertex] (\point) circle (1.6pt);
            }
            \foreach \point in {A,B,C}{
                \fill[black] (\point) circle (2.1pt);
            }
        \end{tikzpicture}}
    }
    \hfill
    \subfigure[Barycentric subdivision of a quadrilateral. Original vertices and edges are shown in black, while newly introduced barycenters and subdivision edges are shown in red and blue, respectively.]
    {
        \resizebox{0.43\textwidth}{!}{\begin{tikzpicture}[line cap=round,line join=round]
            \path[use as bounding box] (0,0) rectangle (4.8,4.0);
            \coordinate (A) at (0.95,3.75);
            \coordinate (B) at (4.05,3.25);
            \coordinate (C) at (4.50,0.20);
            \coordinate (D) at (0.30,0.60);
            \foreach \first/\second/\midpoint in {A/B/AB,B/C/BC,C/D/CD,D/A/DA}{
                \path (\first) -- coordinate[midway] (\midpoint) (\second);
            }
            \coordinate (G) at (barycentric cs:A=1,B=1,C=1,D=1);
            \foreach \point in {A,B,C,D,AB,BC,CD,DA}{
                \draw[bary new edge] (G) -- (\point);
            }
            \draw[bary original edge] (A) -- (B) -- (C) -- (D) -- cycle;
            \foreach \point in {AB,BC,CD,DA,G}{
                \filldraw[bary new vertex] (\point) circle (1.6pt);
            }
            \foreach \point in {A,B,C,D}{
                \fill[black] (\point) circle (2.1pt);
            }
        \end{tikzpicture}}
    }
    \endgroup
    \caption{Examples of barycentric subdivisions.}
    \label{fig:barycentric-subdivision}
\end{figure}

We next bound the mesh.  Let $P\in\mathcal P_r(\sigma)$.  For each $i$, the coordinate $\lambda_i$ varies on $P$ within an interval between two consecutive integers.  Hence, for any $z,z'\in P$,
\[
    |\lambda_i(z)-\lambda_i(z')|\le1
    \quad(0\le i\le s).
\]
Therefore
\[
    \|z-z'\|_2
    =
    \left\|
        \sum_{i=0}^s
        \bigl(\lambda_i(z)-\lambda_i(z')\bigr)v_i
    \right\|_2
    \le
    \sum_{i=0}^s
    |\lambda_i(z)-\lambda_i(z')|\,\|v_i\|_2
    \le s+1
    \le D+1,
\]
because every vertex $v_i$ of the canonical cone triangulation is either $0$ or a standard basis vector.  Every simplex of $\Gamma_r$ is contained in one polytope of $\mathcal P_r$, so the third property holds, for example with $M=D+1$.

It remains to count the vertices.  Again fix an $s$-dimensional simplex $\sigma$.  The relative interior of every polytope of $\mathcal P_r(\sigma)$ is described, for each $i$, by one of the conditions
\[
    \lambda_i=m_i
    \qquad\text{or}\qquad
    m_i<\lambda_i<m_i+1
\]
for some integer $m_i\ge0$.  If such a cell is nonempty, then
\[
    \sum_{i=0}^s m_i
    \le r
    <
    \sum_{i=0}^s m_i+s+1,
\]
because $\sum_i\lambda_i=r$.  Hence $\sum_i m_i$ can take only $s+1$ possible integer values.  For each such value, the number of nonnegative integer tuples $(m_0,\ldots,m_s)$ is $O_s(r^s)$, and for each tuple there are at most $2^{s+1}$ choices of which coordinates are fixed on grid hyperplanes.  Therefore $\mathcal P_r(\sigma)$ has $O_s(r^s)$ polytopes in total.

The vertices of the barycentric subdivision $\Gamma_r$ are precisely the barycenters of the nonempty polytopes of $\mathcal P_r$.  Note that the canonical triangulation of $C(K)$ has $O_K(1)$ simplices, all of dimension at most $D$. Summing the $O_s(r^s)$ bound over these simplices gives
\[
|V(\Gamma_r)|=O_K(r^D).
\]
The same argument on the $(D-1)$-dimensional base gives
\[
    |V(\Gamma_r)\cap r|K||=O_K(r^{D-1}).
\]

For the matching lower bound, choose a $D$-dimensional simplex of the canonical cone triangulation.  Every grid point with barycentric coordinates
\[
    \lambda_i=m_i,
    \qquad
    m_i\in\Z_{\ge0},
    \qquad
    \sum_{i=0}^D m_i=r
\]
is a zero-dimensional polytope of $\mathcal P_r$, and hence a vertex of $\Gamma_r$.  The number of such points is
\[
    \binom{r+D}{D}
    =
    \Omega_D(r^D).
\]
Together with the upper bound, this gives
\[
    |V(\Gamma_r)|=\Theta_K(r^D),
\]
and proves the fourth property.
\end{proof}

\section{Proof of \Cref{lem:prime-conjugacy}}
\label{sec:monomial-reduction}

We restate \Cref{lem:prime-conjugacy} here.

\PrimeOrderConjugacy*

We first introduce a complex whose faces consist of words that differ in every coordinate\footnote{After identifying each alphabet with a set of consecutive integers, this is the higher-dimensional chessboard complex $\Delta_{|A_1|,\ldots,|A_k|}$ of Bj\"orner, Lov\'asz, Vre\'cica, and \v{Z}ivaljevi\'c. \cite{bjorner1994chessboard}}.

\begin{definition}[Disagreement Complex]
\label{def:disagreement-complex}
Let $k\ge1$, and let $A_1,\ldots,A_k$ be finite nonempty alphabets. The \emph{disagreement complex} on these alphabets is the ASC
\[
    Q(A_1,\ldots,A_k)
    :=\left\{\sigma\subseteq\prod_{i=1}^k A_i:
       \forall x,y\in\sigma\ (x\ne y),\ \forall i\in[k],\ x_i\ne y_i\right\}.
\]
\end{definition}

Let $I:=I(A_1,\ldots,A_k)$ and $Q:=Q(A_1,\ldots,A_k)$. For distinct vertices $x,y$, exactly one of $I$ and $Q$ contains $\{x,y\}$ as a face. In both complexes, a vertex set forms a face if and only if all its two-element subsets are faces. Therefore, since the two complexes have the same vertex set, they have the same automorphism group as permutation groups on that vertex set, i.e., $\operatorname{Aut}(I)=\operatorname{Aut}(Q).$

For $x\in V(Q)$, define the neighboring set of $x$ in Q as
\[
    B_Q(x):=\{y\in V(Q)\setminus\{x\}:\{x,y\}\in Q\}
           =\prod_{i=1}^k(A_i\setminus\{x_i\}).
\]
For any vertex set $S\subseteq V(Q)$, write $S^\perp:=\bigcap_{x\in S}B_Q(x)$ and $S^{\perp\perp}:=(S^\perp)^\perp$. For every $f\in\operatorname{Aut}(Q)$,
\[
\begin{aligned}
v\in B_{Q}(x)
&\iff v\neq x\ \text{and}\ \{v,x\}\in Q\\
&\iff f(v)\neq f(x)\ \text{and}\ \{f(v),f(x)\}\in Q\\
&\iff f(v)\in B_{Q}(f(x)).
\end{aligned}
\]
Thus $f(B_{Q}(x))=B_{Q}(f(x))$. This further gives
\begin{equation}
\label{eq:disagreement-perp}
    f(S^\perp)=f(S)^\perp,
    \qquad f(S^{\perp\perp})=f(S)^{\perp\perp}.
\end{equation}

Denote the $P$-action as $\rho$ and $\rho_g:=\rho(g,\cdot)$. Since $\rho_g$ is a simplicial automorphism of $Q(A_1,\ldots,A_k)$, we first analyze the structure of the simplicial automorphisms of $Q(A_1,\ldots,A_k)$.

\paragraph{No binary alphabets.}
We first consider the case when each alphabet has at least three letters. For words $x,y$ in $A_1\times\ldots\times A_k$, denote their Hamming distance as
\[
    d_{\mathrm H}(x,y):=|\{i\in[k]:x_i\ne y_i\}|.
\]

\begin{lemma}
\label{lem:nonbinary-disagreement-aut}
Let $k\ge1$, and let $A_1,\ldots,A_k$ be finite alphabets with $|A_i|\ge3$ for every $i\in[k]$. Every simplicial automorphism of $Q(A_1,\ldots,A_k)$ is monomial.
\end{lemma}

\begin{proof}
Let $W:=A_1\times\ldots\times A_k$ and $Q:=Q(A_1,\ldots,A_k)$. Given $f\in\operatorname{Aut}(Q)$.

For every $x,y\in W$, we have $\{x,y\}^\perp =\prod_{i=1}^k(A_i\setminus\{x_i,y_i\})$ and
\[
    \{x,y\}^{\perp\perp}=\bigcap_{z\in{\{x,y\}^{\perp}}}\prod_{i=1}^k(A_i\setminus\{z_i\})=\prod_{i=1}^k\{x_i,y_i\}.
\]
Consequently, we have
\[
|\{x,y\}^{\perp\perp}|=2^{d_{\mathrm H}(x,y)}.
\]
By \eqref{eq:disagreement-perp}, $f$ preserves Hamming distance, i.e., $d_{\mathrm H}(f(x),f(y))=d_{\mathrm H}(x,y)$.

We next show that any bijection of $W$ preserving Hamming distance is monomial. For $i\in[k]$ and $a\in A_i$, define
\[
    W_{i,a}:=\{x\in W:x_i=a\}.
\]

Fix $i\in[k]$ and $a\in A_i$. Choose words $u,v\in W$ that agree outside coordinate $i$, with $u_i=a$ and $v_i\ne a$. Since $f$ preserves Hamming distance, $f(u)$ and $f(v)$ differ in exactly one coordinate, denoted by $j$. For every $x\in W$, we have
\[
\begin{aligned}
    x\in W_{i,a}
    &\iff d_{\mathrm H}(x,v)=d_{\mathrm H}(x,u)+1\\
    &\iff d_{\mathrm H}(f(x),f(v))=d_{\mathrm H}(f(x),f(u))+1\\
    &\iff f(x)_j=f(u)_j.
\end{aligned}
\]
Therefore $f(W_{i,a})=W_{j,f(u)_j}$ for every $u\in W_{i,a}$. Hence we define $\pi(i)=j$ and $\beta_i(a)=f(u)_j$.

For fixed $i$, the sets $W_{i,a}$, with $a\in A_i$, are pairwise disjoint and partition $W$. Thus these images are pairwise disjoint. This further gives that $\beta_i$ is injective. Their union is $W$, so $\beta_i$ is also surjective. Thus $\beta_i$ is a bijection.

We show that $\pi$ is a permutation of $[k]$. Suppose, for contradiction, that $\pi(i)=\pi(i')$ for distinct $i,i'$. By surjectivity of $\beta_i$, choose $a\in A_i$ and $a'\in A_{i'}$ such that $\beta_i(a)\ne\beta_{i'}(a')$. Note that two sets $W_{i,a}$ and $W_{i',a'}$ always intersect when $i\ne i'$, whereas their images are disjoint, which is a contradiction. Therefore $\pi$ is injective and hence a permutation. Therefore, $f$ is monomial.
\end{proof}

\paragraph{Binary alphabets.}
We next consider the case when there are binary alphabets. We denote the $k$ alphabets as $\underbrace{\{0,1\},\ldots,\{0,1\}}_{b\text{ alphabets}},B_1,\ldots,B_s$ where $b+s=k$ and $|B_i|\ge3$ for every $i\in[s]$. We allow $s=0$, in which case $U$ consists of the empty word and its only bijection is monomial.

\begin{lemma}
\label{lem:one-binary-disagreement-aut}
Let $s\geq0$, and let $B_1,\ldots,B_s$ be finite alphabets with $|B_i|\ge3$ for every $i\in[s]$. Every simplicial automorphism $f$ of $Q(\{0,1\},B_1,\ldots,B_s)$ is monomial and for every $\varepsilon\in\{0,1\}$, $x\in B_1\times\ldots\times B_s$,
\[
    f(\varepsilon,x)=(\varepsilon+\delta,\phi(x)),
\]
where $\delta\in\{0,1\}$, and $\phi$ is a monomial bijection of $B_1\times\ldots\times B_s$.
\end{lemma}

\begin{proof}
For $s=0$, the assertion follows because $f$ permutes the two vertices. Suppose $s\ge1$, and write $Q:=Q(\{0,1\},B_1,\ldots,B_s)$, $Q':=Q(B_1,\ldots,B_s)$ and $W':=B_1\times\ldots\times B_s$.

For any $(\varepsilon,x),(\varepsilon',y)\in V(Q)$, we have
\[
    B_Q((\varepsilon,x))\cap B_Q((\varepsilon',y))\ne\varnothing
    \quad\Longleftrightarrow\quad \varepsilon=\varepsilon'.
\]
Indeed, the binary coordinate of a word in this intersection must be different from both $\varepsilon$ and $\varepsilon'$. If these are equal, each remaining coordinate can be chosen outside $\{x_i,y_i\}$, since $|B_i|\ge3$. By $f(B_{Q}(\varepsilon,x))=B_{Q}(f(\varepsilon,x))$, we have
\[
\begin{aligned}
    u,v\in \{\varepsilon\}\times W'
    &\iff B_Q((\varepsilon,x))\cap B_Q((\varepsilon,y))\ne\varnothing\\
    &\iff B_Q(f(\varepsilon,x))\cap B_Q(f(\varepsilon,y))\ne\varnothing\\
    &\iff f(x),f(y)\in\{\varepsilon'\}\times W'.
\end{aligned}
\]
Thus $f$ permutes the two slices $\{0\}\times W'$ and $\{1\}\times W'$, so there are $\delta\in\{0,1\}$ and bijections $\phi_0,\phi_1$ of $W'$ such that for every $\varepsilon\in\{0,1\}$ and $x\in W'$
\[
    f(\varepsilon,x)=(\varepsilon+\delta,\phi_\varepsilon(x)).
\]
For every $\varepsilon\in\{0,1\}$ and $x,y\in W'$, we have
\[
    \{(\varepsilon,x),(\varepsilon,y)\}^{\perp\perp}=\{\varepsilon\}\times\{x,y\}^{\perp\perp}
       =\{\varepsilon\}\times\prod_{i=1}^s\{x_i,y_i\}.
\]
By \eqref{eq:disagreement-perp}, each $\phi_\varepsilon$ preserves $2^{d_{\mathrm H}(x,y)}$, and hence preserves Hamming distance. By the same argument in the proof of \Cref{lem:nonbinary-disagreement-aut}, $\phi_0$ and $\phi_1$ are monomial.

It remains to prove $\phi_0=\phi_1$. Since $B_Q((0,x))=\{1\}\times B_{Q'}(x)$ and $f(B_{Q'}(x))=B_{Q'}(f(x))$,
\[
    B_{Q'}(\phi_0(x))=\phi_1(B_{Q'}(x))=B_{Q'}(\phi_1(x)).
\]
Hence $\phi_0=\phi_1$. Taking $\phi:=\phi_0=\phi_1$ gives the stated form.
\end{proof}

\begin{lemma}
\label{lem:binary-slice-aut}
Let $b\ge2$, and let $\bar z$ denote the coordinatewise complement of $z\in\{0,1\}^b$. Let $s\geq0$, and let $B_1,\ldots,B_s$ be finite alphabets with $|B_i|\ge3$ for every $i\in[s]$. Every simplicial automorphism $f$ of $Q(\underbrace{\{0,1\},\ldots,\{0,1\}}_{b\text{ alphabets}},B_1,\ldots,B_s)$ has the form such that, for every $z\in\{0,1\}^b, x\in B_1\times\ldots\times B_s$,
\[
    f(z,x)=(\theta(z),\phi_z(x)),
\]
where $\theta$ is a permutation of $\{0,1\}^b$ satisfying $\theta(\bar z)=\overline{\theta(z)}$, and $\phi_z=\phi_{\bar z}$ is a monomial bijection of $B_1\times\ldots\times B_s$.
\end{lemma}

\begin{proof}
Write $Q:=Q(\underbrace{\{0,1\},\ldots,\{0,1\}}_{b\text{ alphabets}},B_1,\ldots,B_s)$, $W':=B_1\times\ldots\times B_s$ and $S_z:=\{z\}\times W'$ for each $z\in\{0,1\}^b$. As in the proof of \Cref{lem:one-binary-disagreement-aut}, for every $(z,x),(z',y)\in V(Q)$,
\[
    B_Q((z,x))\cap B_Q((z',y))\ne\varnothing
    \quad\Longleftrightarrow\quad z=z'.
\]
Thus by a similar argument in the proof of \Cref{lem:one-binary-disagreement-aut}, $f$ permutes the slices $S_z$, giving a permutation $\theta$ and bijections $\phi_z$ such that $f(z,x)=(\theta(z),\phi_z(x))$.

Since $f$ is a simplicial automorphism of $Q$ and $\{(z,x),(\bar z,y)\}\in Q$, we have
\[
f(\{(z,x),(\bar z,y)\})=
\left\{
\bigl(\theta(z),\phi_z(x)\bigr),
\bigl(\theta(\bar z),\phi_{\bar z}(y)\bigr)
\right\}\in Q.
\]
This further gives $\theta(\bar z)=\overline{\theta(z)}$ by \Cref{def:disagreement-complex}.

Fix $z$. Identify $z,\bar z$ with $0,1$, respectively, in the domain, and identify $\theta(z),\theta(\bar z)$ with $0,1$ in the image. Under these identifications, the restriction of $f$ to $S_z\cup S_{\bar z}$ is an automorphism of the complex $Q(\{0,1\},B_1,\ldots,B_s)$. Thus \Cref{lem:one-binary-disagreement-aut} shows that $\phi_z$ is a monomial bijection of $W'$ and $\phi_z=\phi_{\bar z}$.
\end{proof}

We next show a lemma used frequently in our proof.

\begin{lemma}
\label{lem:free-no-invariant-face}
Let a finite group $P$ act simplicially on an ASC $K$ and freely on $|K|$. For every $g\in P\setminus\{e\}$ and every nonempty face $\sigma\in K$, we have $g\sigma\ne\sigma$.
\end{lemma}

\begin{proof}
Suppose, for contradiction, that $g\sigma=\sigma$. Let $c:=\frac{1}{|\sigma|}\sum_{x\in\sigma} x$ denote the barycenter of $\sigma$. Thus
\[
gc=\frac{1}{|\sigma|}\sum_{x\in\sigma}g x=\frac{1}{|\sigma|}\sum_{y\in g\sigma} y=c.
\]
Since $c\in|K|$, this contradicts with the freeness of the $P$-action on $|K|$.
\end{proof}

\begin{proof}[Proof of \Cref{lem:prime-conjugacy}]
Write $P=\langle g\rangle$, where $g$ has prime order $p$, and let $Q:=Q(A_1,\ldots,A_k)$. Then we have $|A_i|\geq2$ for every $i\in[k]$. Otherwise, suppose that $|A_j|=1$ for some $j\in[k]$. Then every pair of vertices agrees in coordinate $j$, so the entire vertex set $V(I)$ forms a face of $I$. By $gV(I)=V(I)$ for every $g\in P$, contradicting  \Cref{lem:free-no-invariant-face}.

\par\smallskip\noindent\textit{Case 1: No binary alphabets.} Suppose $|A_i|\ge3$ for every $i\in[k]$. Since $\operatorname{Aut}(I)=\operatorname{Aut}(Q)$, \Cref{lem:nonbinary-disagreement-aut} shows that every element of $P$ acts by a monomial automorphism on $I$. Thus the given action $\rho$ is already monomial, and we take $F=\id$.

\par\smallskip\noindent\textit{Case 2: At least one binary alphabet.} Write $V(I)=\{0,1\}^b\times W'$, $W':=B_1\times\cdots\times B_s$, where $b\ge1$, $b+s=k$, and $|B_i|\ge3$ for every $i\in[s]$. For each $z\in\{0,1\}^b$, let $\bar z$ denote its coordinatewise complement and define $S_z:=\{z\}\times W'$.

Denote $\rho_g=\rho(g,\cdot)\in \operatorname{Aut}(I)$. Since $\operatorname{Aut}(I)=\operatorname{Aut}(Q)$, \Cref{lem:one-binary-disagreement-aut,lem:binary-slice-aut} imply that $\rho_g$ permutes the $2^b$ slices $\{S_z:z\in\{0,1\}^b\}$. Since each $S_z$ is a nonempty face of $I$, by \Cref{lem:free-no-invariant-face}, $gS_z\ne S_z$. As $g$ has prime order $p$, every orbit of slices has size $p$. Thus $p\mid 2^b$, so $p=2$.

For every vertex $v$, $g\{v,gv\}=\{v,gv\}$ because $g^2=e$. If $v$ and $gv$ agreed in any coordinate, this set would be a nonempty face of $I$, contrary to \Cref{lem:free-no-invariant-face}. Hence $v$ and $gv$ differ in every coordinate. In particular, $g(S_z)=S_{\bar z}$. By \Cref{lem:one-binary-disagreement-aut,lem:binary-slice-aut}, we can therefore write
\[
    g(z,x)=(\bar z,\phi_z(x)),
\]
where $\phi_z$ is a monomial bijection of $W'$ and $\phi_z=\phi_{\bar z}$. If $s=0$, then $g(z)=\bar z$, so the action is coordinatewise and we take $F=\id$. Henceforth assume $s\ge1$.

We first show that each $\phi_z$ fixes the coordinate positions. Let $\phi_z$ consist of coordinate permutation $\pi$ and alphabet bijections  $\beta_i:B_i\to B_{\pi(i)}$. Suppose, for contradiction, that $\pi(i)\ne i$ for some $i\in[k]$. Choose $x\in W'$ with $x_j=\beta(x_i)$. Then $(\phi_z(x))_j=x_j$, contradicting the fact that $(z,x)$ and $g(z,x)$ differ in every coordinate. Therefore $\pi(i)= i$ for every $i\in[k]$, and we can write
\[
    g(z,x_1,\ldots,x_s)
    =(\bar z,\sigma_{z,1}(x_1),\ldots,\sigma_{z,s}(x_s)),
\]
where $\sigma_{z,i}=\sigma_{\bar z,i}$.

We now construct $F$. Fix one element $z_0\in\{0,1\}^b$, and set $\tau_i:=\sigma_{z_0,i}$ for every $i\in[s]$. Both $\sigma_{z,i}$ and $\tau_i$ partition $B_i$ into pairs. Each pair has the form $(a,\sigma_{z,i}(a))$. Matching these pairs gives a bijection $\eta_{z,i}:B_i\to B_i$ such that $\eta_{z,i}=\eta_{\bar z,i}$ and
\[
    \eta_{z,i}\sigma_{z,i}\eta_{z,i}^{-1}=\tau_i.
\]
For every $(z,x_1,\ldots,x_s)\in\{0,1\}^b\times W'$, define
\[
    F(z,x_1,\ldots,x_s)
       :=(z,\eta_{z,1}(x_1),\ldots,\eta_{z,s}(x_s)),
\]
Since each $\eta_{z,i}$ is a bijection, so $F$ is well-defined and bijective.

For every two words $w_1:=(z,x_1,\ldots,x_s)$ and $w_2:=(z',y_1,\ldots,y_s)$, they form a face of $Q$ if and only if $z'=\bar z$ and $x_i\ne y_i$ for all $i\in[s]$. Since $\eta_{z,i}$ is a bijection and $\eta_{z,i}=\eta_{\bar z,i}$, $F(w_1)$ and $F(w_2)$ form a face if and only if $w_1$ and $w_2$ form a face. Since the faces of $Q$ are determined by their two-element subsets, $F$ is a simplicial automorphism of $Q$, and hence of $I$.

Finally, for every $(z,x_1,\ldots,x_s)$,
\[
\begin{aligned}
    FgF^{-1}(z,x_1,\ldots,x_s)
    &=\bigl(\bar z,
       (\eta_{z,i}\sigma_{z,i}\eta_{z,i}^{-1}(x_i))_{i=1}^s\bigr)\\
    &=(\bar z,\tau_1(x_1),\ldots,\tau_s(x_s)).
\end{aligned}
\]
The map in each coordinate is independent, so the conjugated action is coordinatewise, and hence monomial.
\end{proof}

\end{document}